\documentclass[a4paper,UKenglish,cleveref, autoref, thm-restate]{lipics-v2019}
\title{Lambek-Grishin Calculus: Focusing, Display and Full Polarization}

\titlerunning{Display, Focusing and Full Polarization}

\author{Giuseppe Greco}
    {Utrecht University, The Netherlands}
    {}
    {https://orcid.org/0000-0002-4845-3821}
    {NWO grant under the scope of the project ``A composition calculus for vector-based semantic modelling with a localization for Dutch'' (360-89-070).}
\author{Valentin D. Richard}
    {\'Ecole Normale Sup\'erieure Paris-Saclay, France}
    {}
    {https://orcid.org/0000-0002-9236-8106}
    {UPS grant \textit{Bourse de mobilit\'e internationale de stage de l’Universit\'e Paris-Salcay}.}
\author{Michael Moortgat}
    {Utrecht University, The Netherlands}
    {}
    {}
    {}
\author{Apostolos Tzimoulis}
    {Vrije Universiteit, The Netherlands}
    {}
    {}
    {}
\authorrunning{G. Greco, V.\,D. Richard, M. Moortgat, and A. Tzimoulis}
\Copyright{Giuseppe Greco, Valentin D. Richard, Michael Moortgat, and Apostolos Tzimoulis}

\nolinenumbers 
\newtheorem{notation}[theorem]{Notation}
\usepackage{cmll}
\usepackage{xspace} 
\usepackage{marginnote}
\usepackage{amsmath}
\usepackage{amssymb}
\usepackage{stmaryrd}
\usepackage{txfonts}
\usepackage{booktabs}
\usepackage{multirow}
\usepackage{colortbl}
\usepackage{bbold}
\usepackage{stmaryrd} % symbols \oplus, \oslash. \obslash
\usepackage{hyperref}
\usepackage{proof}
\usepackage{bussproofs}
\EnableBpAbbreviations
\usepackage{version}
\usepackage{mathtools}
\usepackage{mathtools}
\usepackage{turnstile}
\usepackage{cmll} % \parr of LL
\usepackage{bbm}
\usepackage{bbold} % bold ciphers
\usepackage{mathdots} % \iddots
\usepackage{graphicx,tikz,tikz-qtree}
\usetikzlibrary{calc} % coordinate calculation
\usetikzlibrary{arrows.meta,decorations.markings} % customized decorations
\usetikzlibrary{arrows}
\usetikzlibrary{patterns}
\usepackage{linguex} % linguistic supplies
\newenvironment{pic}[1]{\hspace*{\fill}
\begin{tikzpicture}[
x=.22ex, y=.2ex, semithick, scale=1, every node/.style={scale=1},
inner sep=0.3ex, outer sep=0pt, minimum size=0ex, label distance=-0.2ex]
\tikzset{pro/.style={postaction=decorate,decoration={name=markings,% custumized decoration
    mark=at position 0.5 with {
    \draw[-] (+0,+5\pgflinewidth) to (+0,-5\pgflinewidth);
    }}}}
\tikzset{outer sep=2pt}
#1
}{\end{tikzpicture}\hspace*{\fill}}

\def\fCenter{\vdash}
\def\mc{\multicolumn}

\newcommand{\sep}{\ \vert\ }
\def\fns{\footnotesize}

\newcommand{\fDLG}{\textbf{fD.LG}\xspace}
\newcommand{\fLG}{\textbf{f.LG}\xspace}

\newcommand{\LG}{\mathbb{LG}}
\newcommand{\FPLG}{\mathbb{FP.LG}}
\renewcommand{\LG}{\mathbb{LG}}
\newcommand{\ite}[4]{%
\left \{%
\begin{array}{ll}%
     #1 & \text{#2}\\%
     #3 & \text{#4}%
\end{array}%
\right.%
}

\newcommand{\graycell}{\cellcolor{gray!25}}

\newcommand{\focus}[1]{\fbox{$#1$}}
\newcommand{\FtoM}[1]{\llceil #1 \rrceil}
\newcommand{\Ftom}[1]{\llfloor #1 \rrfloor}
\newcommand{\TtoF}[1]{\lfloor #1 \rfloor}

\newcommand{\FtoT}[2][\rho]{\lceil #2 \rceil#1}
\newcommand\val[1]{{\lbrack\!\lbrack} {#1}{\rbrack\!\rbrack}}

\def\fel{\leftharpoonup}
\def\fer{\rightharpoonup}
\def\fil{\leftharpoondown}
\def\fir{\rightharpoondown}
\newcommand{\buar}{\textcolor{blue}{{\upharpoonleft}}}
\newcommand{\BUARR}{\textcolor{blue}{\hat{\upharpoonleft}}}
\newcommand{\sbuar}{\textcolor{blue}{\uparrow}}
\newcommand{\SBUARR}{\textcolor{blue}{\hat{\uparrow}}}

\newcommand{\rdar}{\textcolor{red}{\downharpoonright}}
\newcommand{\RDARR}{\textcolor{red}{\check{\downharpoonright}}}

\newcommand{\srdar}{\textcolor{red}{\downarrow}}
\newcommand{\SRDARR}{\textcolor{red}{\check{\downarrow}}}

\newcommand{\mand}{\otimes}
\newcommand{\mdrarr}{\varobslash}
\newcommand{\mdlarr}{\varoslash}
\newcommand{\mor}{\oplus}
\newcommand{\mrarr}{\,\backslash\,}
\newcommand{\mlarr}{\,\slash\,}
\newcommand{\mandr}{\otimes_r}
\newcommand{\mdrarrr}{\varobslash_r}
\newcommand{\mdlarrr}{\varoslash_r}
\newcommand{\morr}{\oplus_r}
\newcommand{\mrarrr}{\,\backslash_r\,}
\newcommand{\mlarrr}{\,\slash_r\,}
\newcommand{\mandl}{\otimes_\ell}
\newcommand{\mdrarrl}{\varobslash_\ell}
\newcommand{\mdlarrl}{\varoslash_\ell}
\newcommand{\morl}{\oplus_\ell}
\newcommand{\mrarrl}{\,\backslash_\ell\,}
\newcommand{\mlarrl}{\,\slash_\ell\,}

\newcommand{\MAND}{\,\hat{\otimes}\,}
\newcommand{\MRARR}{\,\check{\backslash}\,}
\newcommand{\MLARR}{\,\check{\slash}\,}
\newcommand{\MOR}{\,\check{\oplus}\,}
\newcommand{\MDRARR}{\,\hat{\varobslash}\,}
\newcommand{\MDLARR}{\,\hat{\varoslash}\,}
\newcommand{\MANDL}{\,\hat{\otimes}_\ell\,}
\newcommand{\MRARRL}{\,\check{\backslash}_\ell\,}
\newcommand{\MLARRL}{\,\check{\slash}_\ell\,}
\newcommand{\MORL}{\,\check{\oplus}_\ell\,}
\newcommand{\MDRARRL}{\,\hat{\varobslash}_\ell\,}
\newcommand{\MDLARRL}{\,\hat{\varoslash}_\ell\,}
\newcommand{\MANDR}{\,\hat{\otimes}_r\,}
\newcommand{\MRARRR}{\,\check{\backslash}_r\,}
\newcommand{\MLARRR}{\,\check{\slash}_r\,}
\newcommand{\MORR}{\,\check{\oplus}_r\,}
\newcommand{\MDRARRR}{\,\hat{\varobslash}_r\,}
\newcommand{\MDLARRR}{\,\hat{\varoslash}_r\,}

\newcommand{\likes}{\texttt{likes}}

\newcommand{\every}{\texttt{every}}
\newcommand{\everyone}{\texttt{everyone}}
\newcommand{\one}{\texttt{one}}
\newcommand{\some}{\texttt{some}}
\newcommand{\teacher}{\texttt{teacher}}

\newcommand{\GP}{\mathring{P}}
\newcommand{\GQ}{\mathring{Q}}
\newcommand{\GR}{\mathring{R}}

\newcommand{\GL}{\mathring{L}}
\newcommand{\GM}{\mathring{M}}
\newcommand{\GN}{\mathring{N}}

\newcommand{\GX}{\mathring{X}}
\newcommand{\GY}{\mathring{Y}}
\newcommand{\GZ}{\mathring{Z}}

\newcommand{\GG}{\mathring{\Gamma}}
\newcommand{\GD}{\mathring{\Delta} \xspace}

\newcommand{\GSi}{\mathring{\Sigma}}
\newcommand{\SP}{\dot{P}}
\newcommand{\SQ}{\dot{Q}}

\newcommand{\SM}{\dot{M}}
\newcommand{\SN}{\dot{N}}

\newcommand{\SX}{\dot{X}}
\newcommand{\SY}{\dot{Y}}

\newcommand{\SG}{\dot{\Gamma}}
\newcommand{\SD}{\dot{\Delta}}

\newcommand{\A}{\mathcal{A}}
\newcommand{\B}{\mathcal{B}}
\newcommand{\C}{\mathcal{C}}

\newcommand{\F}{\mathcal{F}}
\newcommand{\G}{\mathcal{G}}

\renewcommand{\S}{\mathcal{S}}
\newcommand{\T}{\mathcal{T}}

\newcommand{\PF}{^{+}_{\textsf{F}}}
\newcommand{\NF}{^{-}_{\textsf{F}}}
\newcommand{\PS}{^{+}_{\textsf{S}}}
\newcommand{\NS}{^{-}_{\textsf{S}}}
\newcommand{\bbA}{\mathbb{A}}

\newcommand{\bbG}{\mathbb{G}}
\newcommand{\N}{\textcolor{blue}{\mathbb{N}}}
\newcommand{\sN}{\textcolor{blue}{\dot{\mathbb{N}}}}
\newcommand{\gN}{\textcolor{blue}{\mathring{\mathbb{N}}}}
\renewcommand{\P}{\textcolor{red}{\mathbb{P}}}
\newcommand{\sP}{\textcolor{red}{\dot{\mathbb{P}}}}
\newcommand{\gP}{\textcolor{red}{\mathring{\mathbb{P}}}}

\newcommand{\Pos}{\mathsf{Pos}}
\newcommand{\Neg}{\mathsf{Neg}}
\newcommand{\Pure}{\mathsf{Pure}}
\newcommand{\Shifted}{\mathsf{Shifted}}
\newcommand{\sort}{\mathsf{sort}}
\newcommand{\Sort}{\mathsf{Sort}}
\newcommand{\sortpst}{\mathsf{sort \text{\textendash} pst}}
\newcommand{\pst}{\mathsf{pst}}
\newcommand{\Pst}{\mathsf{Pst}}
\newcommand{\ar}{\mathsf{ar}}

\renewcommand{\Form}{\mathsf{Fm}}
\newcommand{\Str}{\mathsf{Str}}

\renewcommand{\t}[1][]{\ t#1 \ }
\newcommand{\nrvd}{\textcolor{red}{\vdash}}
\newcommand{\rvd}{\ \,\textcolor{red}{\vdash}\ \,}
\newcommand{\nrvvd}{\textcolor{red}{\dot{\Vdash}}}
\newcommand{\rvvd}{\ \,\textcolor{red}{\dot{\Vdash}}\ \,}

\newcommand{\urvvd}{\ \,\textcolor{red}{\text{\d{$\Vdash$}}}\ \,} % u for underivable
\newcommand{\nrvvvd}{\textcolor{red}{\Vvdash}}
\newcommand{\rvvvd}{\ \,\textcolor{red}{\Vvdash}\ \,}
\newcommand{\nrvD}{\textcolor{red}{\mathring{\vdash}}}
\newcommand{\rvD}{\ \,\textcolor{red}{\mathring{\vdash}}\ \,}
\newcommand{\nbvd}{\textcolor{blue}{\vdash}}
\newcommand{\bvd}{\ \, \textcolor{blue}{\vdash}\ \,}
\newcommand{\nbvvd}{\textcolor{blue}{\text{\d{$\Vdash$}}}}
\newcommand{\bvvd}{\ \,\textcolor{blue}{\text{\d{$\Vdash$}}}\ \,}

\newcommand{\ubvvd}{\ \,\textcolor{blue}{\dot{\Vdash}}\ \,}
\newcommand{\nbvvvd}{\textcolor{blue}{\Vvdash}}
\newcommand{\bvvvd}{\ \,\textcolor{blue}{\Vvdash}\ \,}
\newcommand{\nbvD}{\textcolor{blue}{\mathring{\vdash}}}
\newcommand{\bvD}{\ \, \textcolor{blue}{\mathring{\vdash}}\ \,}
\newcommand{\nvd}{\textcolor{black}{\vdash}}
\newcommand{\vd}{\ \,\textcolor{black}{\vdash}\ \,}
\newcommand{\nvvd}{\textcolor{black}{\text{\d{$\Vdash$}}}} 
\newcommand{\tvvd}{\ \,\textcolor{black}{\text{\d{$\Vdash$}}}\ \,} % t for turnstile
\newcommand{\vvdn}{\textcolor{black}{\dot{\Vdash}}}
\newcommand{\vvdt}{\ \,\textcolor{black}{\dot{\Vdash}}\ \,}
\newcommand{\nvvvd}{\textcolor{black}{\Vvdash}}
\newcommand{\vvvd}{\ \,\textcolor{black}{\Vvdash}\ \,}
\newcommand{\nvD}{\textcolor{black}{\mathring{\vdash}}}
\newcommand{\vD}{\ \,\textcolor{black}{\mathring{\vdash}}\ \,}
\newcommand{\proto}{{\relbar\joinrel\mapstochar\joinrel\rightarrow}}
\newcommand{\wrel}{\preccurlyeq}

\newcommand*{\leqqq}{\mathrel{\vcenter{\offinterlineskip\hbox{\raisebox{0.16em}{$\:\scriptstyle\mkern-1.3mu<$}}\vskip-.30ex\hbox{$\equiv$}}}}
\newcommand{\eqql}{\raisebox{\depth}{\scalebox{1}[-1]{$\preceqq$}}}
\newcommand{\arvd}{\,\textcolor{red}{\leq}\,}
\newcommand{\hrvvd}{\,\textcolor{red}{\eqql}\,}

\newcommand{\arvvvd}{\,\textcolor{red}{\leqqq}\,}
\newcommand{\arvD}{\,\textcolor{red}{\mathring{\leq}}\,}
\newcommand{\abvd}{\,\textcolor{blue}{\leq}\,}
\newcommand{\hbvvd}{\,\textcolor{blue}{\preceqq}\,}

\newcommand{\abvvvd}{\,\textcolor{blue}{\leqqq}\,}
\newcommand{\abvD}{\,\textcolor{blue}{\mathring{\leq}}\,}
\newcommand{\hvd}{\preceq}

\newcommand{\hvvd}{\preceqq}

\newcommand{\vvdh}{\,\eqql\,}

\newcommand{\hvD}{\,\mathring{\hvd}\,}

\newcommand{\Pa}{^{\partial}}

\def\felfir{\leftrightharpoons}
\def\ferfil{\rightleftharpoons}

\begin{document}
\maketitle
\begin{abstract}
\emph{Focused sequent calculi} are a refinement of sequent calculi, where additional side-conditions on the applicability of inference rules force the implementation of a proof search strategy. Focused cut-free proofs exhibit a special normal form that is used for defining identity of sequent calculi proofs. We introduce a novel focused display calculus \fDLG and a fully polarized algebraic semantics $\FPLG$ for Lambek-Grishin logic by generalizing the theory of \emph{multi-type calculi} and their algebraic semantics with \emph{heterogenous consequence relations}. The calculus \fDLG has \emph{strong focalization} and it is \emph{sound and complete} w.r.t.~$\FPLG$. This completeness result is in a sense stronger than completeness with respect to standard polarized algebraic semantics (see e.g.~the phase semantics of Bastenhof for Lambek-Grishin logic or Hamano and Takemura for linear logic), insofar we do not need to quotient over proofs with consecutive applications of shifts over the same formula. We plan to investigate the connections, if any, between this completeness result and the notion of \emph{full completeness} introduced by Abramsky et al. We also show a number of additional results. \fDLG is sound and complete w.r.t.~LG-algebras: this amounts to a semantic proof of the so-called \emph{completeness of focusing}, given that the standard (display) sequent calculus for Lambek-Grishin logic is complete w.r.t.~LG-algebras. \fDLG and the focused calculus \fLG of Moortgat and Moot are equivalent with respect to proofs, indeed there is an effective translation from \fLG-derivations to \fDLG-derivations and vice versa: this provides the link with operational semantics, given that every \fLG-derivation is in a Curry-Howard correspondence with a directional $\overline\lambda\mu\widetilde{\mu}$-term.  

%Focused display calculi, by making use of side conditions on the applicability of inference rules, achieve two remaining goals: reducing proof-search space and providing a synthetic normal form of cut-free proofs.In this paper, we aim at internalizing focus side conditions in a multi-sorted proper display calculus. Following the idea of a categorical semantics for multi-modal Lambek calculus, we have to distinguish formulas vs. structures, input vs output structures and positive vs negative formulas. Key features of our approach are the following: (i) the proof-theoretic relevant distinctions are captured by different sorts, (ii) each sort is interpreted in a specific category, (iii) the shift operators (changing the polarity of a formula from positive to negative or vice versa) are introduced as modalities in the calculus and (iv) a general categorical framework with distributors is provided. The verification that each cut-free proof has the so-called synthetic normal form, what is key to show completeness, is work in progress.
\end{abstract}

{\renewcommand{\L}{\mathcal{L}}

%%%%%%%%%%
\section{Introduction}
\label{sec:introduction}

The problem of the identity of proofs is a fundamental one. It has been actively investigated in philosophy and  mathematics (when do two proofs correspond to the same argument?), and in computer sciences (when do two algorithms correspond to the same program?). A logic can be presented by different formalisms. Sequent calculi often exhibit syntactically different proofs of the very same end-sequent. Some of these proofs differ from each other by trivial permutations of inference rules. Other formalisms, like natural deduction calculi or proof nets, are less sensitive to inference rule permutations and are usually taken as benchmarks for defining identity of proofs. \emph{Focused sequent calculi} \cite{And92,Andreoli:2001,Mil04} make use of syntactic restrictions on the applicability of inference rules achieving three main goals: (i) the proof search space is considerably reduced without losing completeness, (ii) every cut-free proof comes in a special normal form, 
%comes in the so-called (strong) \emph{focalized normal form} (sometimes referred to as synthetic normal form), 
(iii) leading to a criterion for defining identity of sequent calculi proofs. Being able to identify or tell apart two proofs has far-reaching consequences. In particular, in the tradition of parsing-as-deduction \cite{lambek1958mathematics,lam61}, various logical systems -- and notably various extensions of the Lambek calculus -- have been proposed to recognise not only whether sentences are syntactically well-formed, but also to capture different semantic readings by `genuinely different' proofs in the type logic \cite{Bernardi--Moortgat:2007,Moortgat--Moot:2011}.

In this paper, we focus on the minimal Lambek-Grishin logic and we provide a novel algebraic and proof-theoretic analysis of the focused Lambek-Grishin calculus \cite{Moortgat--Moot:2011}. More in general, this analysis leads to the identification of a new class of display calculi and their natural algebraic semantics. The gist of the analysis is to generalise (and refine) \emph{multi-type display calculi} and \emph{heterogeneous algebras} (\cite{birkhoff1970heterogeneous}) admitting not only heterogeneous operators, but also \emph{heterogeneous consequence relations}, now interpreted as \emph{weakening relations} \cite{Kurz--Moshier--Jung:2019} (i.e.~a natural generalisation of partial orders). Here, we introduce a specific instance of this class tailored for the signature of the Lambek-Grishin logic and we plan to provide the full picture as future work. In particular, we plan to show that if a calculus belongs to this class, then it enjoys cut-elimination, aiming at generalizing the cut-elimination meta-theorem in the tradition of display calculi (see \cite{Wansing:2002}). Moreover, we conjecture that any \emph{displayable logic} (\cite{GMPTZ}) can be equivalently presented as an instance of this class. The next paragraph summarises the main features of this analysis in general terms, without special reference to Lambek-Grishin logic. 

In the case of focused sequent calculi, the distinction between \emph{positive} versus \emph{negative} formulas is the key ingredient for organising proofs in so-called \emph{phases}. The distinction is proof-theoretically relevant in that it reflects a fundamental distinction between logical introduction rules: the left introduction rules for positive connectives are \emph{invertible} where the right introduction rules are \emph{non-invertible} in general, and vice versa for negative connectives. We observe that this distinction is also semantically grounded, indeed positive formulas (in the original language of the logic) are left adjoints and negative formulas (in the original language of the logic) are right adjoints. Proofs in \emph{focalized normal form} (see \cite{Moortgat--Moot:2011}) are cut-free proofs organised in three phases: two focused phases (either positive or negative) and non-focused phases (also called neutral phases). A focused positive (resp.~negative) phase in a derivation is a proof-section (see definition \ref{def:proof-section}) where a formula is decomposed as much as possible only by means of non-invertible logical rules for positive (resp.~negative) connectives. This formula and all its immediate subformulas in this proof-section are said `in focus'. All the other rules are applied only in non-focused phases. So, each derivable sequent has at most one formula in focus. Moreover, the interaction between two focused phases is always mediated by a non-focused phase.

So-called \emph{shift operators} -- usually denoted as $\uparrow$ and $\downarrow$ (\cite{Hamano--Scott:2007,Hamano--Takemura:2010,Bastenhof:2011}) -- are often considered to polarize a focused sequent calculus, i.e.~as a tool to control the interplay between positive and negative formulas and the interaction between phases. Shifts are adjoint unary operators that change the polarity of a formula, where $\uparrow$ goes from positive to negative, $\downarrow$ goes from negative to positive, and $\uparrow \,\dashv\, \downarrow$. In this paper, we consider positive and negative formulas as formulas of different sorts.\footnote{Note in the literature on multi-type display calculi `types' is used instead of `sorts'.} We also distinguish between positive (resp.~negative) \emph{pure} formulas and positive (resp.~negative) \emph{shifted} formulas, i.e.~formulas under the scope of a shift operator, hence we call it a \emph{full} polarization. So, we end up in considering four different sorts, each of which is interpreted in a different sub-algebra. Therefore, in this setting shifts are heterogeneous operators, where $\uparrow$ gets split into $\sbuar$ (from positive pure formulas into negative shifted formulas) and $\buar$ (from positive shifted formulas into negative pure formulas), $\downarrow$ gets split into $\srdar$ (from negative pure formulas into positive shifted formulas) and $\rdar$ (from negative shifted formulas into positive pure formulas), $\sbuar \dashv \rdar$ and $\buar \dashv \srdar$. Moreover, the composition of two shifts is still either a closure or an interior operator (by adjunction), but we do not assume that it is an identity. 

The present investigation was largely inspired by the work of Samson Abramsky, specifically from his contributions on the Geometry of Interactions program (see \cite{Abramsky--Jagadeesan:1994a}) and his categorical and game-theoretic perspective on the semantics of linear logic. In particular, in \cite{Abramsky--Jagadeesan:1994b} Abramsky and Jagadeesan introduce a game-theoretic semantics for the multiplicative fragment of linear logic expanded with the so-called MIX rule, and show a strong form of completeness they call \emph{full completeness}. This notion is inherently stronger than standard completeness (that is with respect to provability), indeed it requires ``a semantic characterization of the space of proofs of a given logic''. More precisely, given a logic \textbf{L} and the appropriate categorical model $\mathcal{M}$ where formulas are interpreted as object and proofs as morphisms, $\mathcal{M}$ is fully complete for $\mathcal{L}$ if every morphism $\val{\pi}\,:\,\val{A} \longrightarrow \val{B}$ in $\mathcal{M}$ is the denotation of some proof $\pi$ of $A \fCenter B$ in \textbf{L}. In \cite{Abramsky--Mellies:1999} Abramsky and Melliès extend the full completeness result to the multiplicative-additive fragment of linear logic with respect to a new concurrent form of games semantics. 

First of all, let us emphasize the striking differences between the present approach and the two papers cited above. Leaving aside that in this paper we focus on the Lambek-Grishin logic, notice that in \cite{Abramsky--Jagadeesan:1994b} and \cite{Abramsky--Mellies:1999} the identity of proofs is addressed via proof nets, where here we work with focused sequent calculi. We do not consider this an essential difference from a conceptual point of view. Most importantly, the notion of full completeness makes sense, strictly speaking, only for a categorical semantics. What we do here is to take seriously the idea that different semantics, even algebraic semantics, could be more tightly or less tightly connected to a logic, in the sense that they reflect more closely or less closely the structure of proofs. Moreover, the notion of weakening relation (see section \ref{sec:WeakeningRelations}), the key ingredient for defining fully polarized algebras, has a natural categorical presentation (see \cite{Kurz--Moshier--Jung:2019}) and we plan to provide a categorical presentation of the present approach in future work, where the notion of full completeness can be rigorously defined.       

The paper is structured as follows. In section \ref{sec:semantics} we introduce fully polarized LG-algebras $\FPLG$. 
%namely a heterogeneous algebraic semantics where different partial orders coexist with different weakening relations. 
In section \ref{sec:calculus} we introduce the focused display calculus \fDLG for the minimal Lambek-Grishin logic and we prove that it has the strong focalization property. In section \ref{sec:completeness-of-focusing}, we show that the calculus \fDLG is sound and complete w.r.t.~$\FPLG$ and LG-algebras, and in section \ref{ap:heterogeneous-dc} that \fDLG has canonical cut-elimination. Section \ref{sec:ProofTranslations} provides the effective translation between derivations of the calculi \fDLG and \fLG.

\section{Algebraic semantics}
\label{sec:semantics}

In this section we first recall the definition of Lambek-Grishin algebras, weakening relations and their properties needed in what follows. Then, we define fully polarized LG-algebras. 
% and show how the weakening relation defined on the collage poset (i.e.~the diagonal arrows in figure \ref{fig:fplg-wr}) arises as the composition of two weakening relations in the algebra (proposition \ref{prop:WeakeningOnCollagePoset})
%Finally, we show that Lambek-Grishin algebras can be equivalently presented as fully polarized LG-algebras (theorem \ref{thm:LGAsFPLG}).

\subsection{Preliminaries}
\label{sec:WeakeningRelations}

\paragraph*{Lambek-Grishin algebras}

The basic Lambek-Grishin logic \textbf{LG} \cite{Moortgat:2009} is the pure logic of residuation in the signature that expands the (non-unital, non-associative) Lambek calculus \cite{lam61} with the so-called Grishin connectives (i.e.~a co-tensor $\mor$ and its residuals $\mdrarr, \mdlarr$). \textbf{LG} is complete w.r.t.~Lambek-Grishin algebras defined below.

\begin{definition}
\label{def:LGAlgebra}
A basic Lambek-Grishin algebra $\bbG = (G, \leq, \mand, \mor, \mrarr, \mdrarr, \mlarr, \mdlarr)$ is a partially ordered algebra endowed with six binary operations compatible with the order $\leq$.
%, where $\mand$ and $\mor$ are order-preserving in both coordinates, $\mrarr$ and $\mdrarr$ are order-reversing in the first and order-preserving in the second coordinate, and $\mlarr, \mdlarr$ are order-preserving in the first and order-reversing in the second coordinate.
Moreover, the following residuation laws hold:
\begin{equation}\label{eq:adj-lg}
    B \leq A \mrarr C ~~~~\text{iff}~~~~
    A \mand B \leq C ~~~~\text{iff}~~~~
    A \leq C \mlarr B ~~~~~~~~
    C \mdlarr B \leq A ~~~~\text{iff}~~~~
    C \leq A \mor B ~~~~\text{iff}~~~~
    A \mdrarr C \leq B
\end{equation}
\end{definition}

%%%%%
\paragraph*{Weakening relations and collages}

In this paper we use weakening relations \cite{JunKegMos99,Kurz--Moshier--Jung:2019,GalJip20,GalJip,BilKurPetVel13} to interpret the heterogeneous consequence relations of the calculus introduced in section \ref{subsec:fDLG}. Weakening relations can be viewed as the order-theoretic equivalents of profunctors \cite{Benabou:1973} (aka distributors or bimodules), which have already been considered in models of polarized logic \cite{Hamano--Scott:2007,Cockett--Seely:2007}. In particular, partial orders are weakening relations where $\A = \B$ and ${\leq_{\A}} =  {\leq_{\B}}$. We use $\wrel\, \subseteq \A \times \B$, $\wrel_{\A}^{\B}$ and $\A\, \proto\, \B$ interchangeably to denote a weakening relation with source $\A$ and target $\B$, and $\wrel_{\A}$ as an abbreviation for $\wrel_{\A}^{\A'}$. Given two relations $R$ and $S$, we use $RS$ to denote composition of relations. 

\begin{definition}
\label{def:WeakeningRelation}
A \textbf{weakening relation} is a relation $\wrel\, \subseteq \A \times \B$ on two partially ordered set $(\A, \leq_{\A})$ and $(\B, \leq_{\B})$ that is \emph{compatible with the orders} $\leq_{\A}$ and $\leq_{\B}$ in the following sense 

\begin{center}
    \AXC{$A' \leq_{\A} A$}
    \AXC{$A \wrel B$}
    \AXC{$B \leq_{\B} B'$}
    \TrinaryInfC{$A' \wrel B'$}
    \DP
\end{center}
\end{definition}
%This property is necessary and sufficient to ensure soundness of cuts. 

% We now recall some notions needed for this paper.

\begin{definition}
Given two weakening relations ${\wrel_{\A}} \subseteq \A \times \A'$ and ${\wrel_{\B}}  \subseteq \B \times \B'$, we say that the order-preserving functions $L : \A \to \B$ and $R: \B' \to \A'$ form a \textbf{heterogeneous adjoint pair} $L \dashv_{\wrel_{\A}}^{\wrel_{\B}} R$ if for every $A \in \A$ and $B' \in \B'$, 

\begin{equation}
\begin{pic}\tikzset{outer sep=0.05em}
\draw (-120,15) node{$L(A) \wrel_{\B} B' \text{  iff  } A \wrel_{\A} R(B')$};
\draw (0,0) node(a'){$\A$};\draw (50,0) node(b'){$\B$};
\draw (0,30) node(a){$\A'$};\draw (50,30) node(b){$\B'$};
\draw[<-] (a) edge[pro] node[left,outer sep=0.4em]{$\wrel_{\A}$} (a');
\draw[<-] (b) edge[pro] node[right,outer sep=0.4em]{$\wrel_{\B}$} (b');
\draw[<-] (a) edge[bend left] node[below]{$R$} (b);
\draw[<-] (b') edge[bend left] node[above]{$L$} (a');
\draw (25,15) node{\rotatebox{-90}{$\vdash$}};
\end{pic}
\end{equation}
\end{definition}

If $\A' = \A$, ${\wrel_{\A}} = {\leq_{\A}}$, $\B' = \B$ and ${\wrel_{\B}} = {\leq_{\B}}$, we recover the usual definition of adjunction.

\begin{proposition}\label{prop:wr}
If $L \dashv_{\wrel_{\A}}^{\wrel_{\B}} R$ is a heterogeneous adjunction, then it defines a weakening relation $\wrel\, \subseteq \A \times \B'$ by $A \wrel B'$ iff $L(A) \wrel_{\B} B'$, which is also equivalent to $A \wrel_{\A} R(B')$. We say that $\wrel$ is \textbf{the weakening relation represented by} $L \dashv_{\wrel_{\A}}^{\wrel_{\B}} R$.
\end{proposition}

\noindent
The proof only requires unfolding of the definitions (see appendix \ref{ap:proofs}).

\begin{definition}\label{def:collage-order}
If $\wrel\, \subseteq \A \times \B$ is a weakening relation, then the relation $\leq_{\A \sqcup \B} \,:=\, \leq_{\A} \sqcup \wrel \sqcup \leq_{\B}$ defined on the disjoint union $\A \sqcup \B$ is an order. We call it the \textbf{collage order} on $\A \sqcup \B$.
\end{definition}

The collage order $(\A \sqcup \B, \leq_{\A \sqcup \B})$ corresponds to the collage \cite{Street:1980} (or cograph) of $\wrel$ seen as a profunctor. We extend the $\sqcup$ notation to weakening relations.

\begin{definition}\label{def:collage-wr}
If we are in the following situation:

\begin{pic}\tikzset{outer sep=0.05em}
\draw (0,40) node(a'){$\A'$};\draw (50,40) node(b'){$\B'$};
\draw (0,0) node(a){$\A$};\draw (50,0) node(b){$\B$};
\tikzset{outer sep=0.4em}
\draw[->] (a) edge[pro] node[left]{$\wrel_{\A}$} (a')
    edge[loop left] node[left,]{$\leq_{\A}$} (a)
    edge[pro] node[above left,]{$\hvd$} (b')
    edge[pro] node[below]{$\vvdh$} (b);
\draw[->] (b) edge[pro] node[right]{$\wrel_{\B}$} (b')
    edge[loop right] node[right]{$\leq_{\B}$} (b);
\draw[->] (a') edge[loop left] node[left]{$\leq_{\A'}$} (a')
    edge[pro] node[above]{$\hvvd$} (b');
\draw[->] (b') edge[loop right] node[right]{$\leq_{\B'}$} (b');
\draw (-19.5,3) -- (-14,0);\draw (-20.5,43) -- (-15,40);
\draw (68.5,-3) -- (63,0);\draw (69.5,37) -- (65,40);
\end{pic}

\noindent and we also have ${\wrel_{\A} \hvvd} \ \subseteq\ \hvd$ and ${\vvdh\! \wrel_{\B}} \ \subseteq\ \hvd$, then the relation $\hvD \,:=\, \hvvd \sqcup \hvd \sqcup \vvdh$ is a weakening relation on the collage orders $\A \sqcup \A'$ and $\B \sqcup \B'$, and we call it the \textbf{collage weakening relation}.
\end{definition}

%%%%%
\subsection{Fully polarized LG-algebra}
\label{subsec:FullyPolarizedLG}

We write $\A\Pa$ for the order dual of $(\A, \leq_{\A})$, i.e. $A \leq_{\A\Pa} A'$ iff $A' \leq_{\A} A$. We use $P,Q$ (resp.~$\SP, \SQ$) for pure (resp.~shifted) positive elements, i.e.~elements in the poset $\P$ (resp.~in $\sP$); $M, N$ (resp.~$\SM, \SN$) for pure (resp.~shifted) negative elements,  i.e.~elements in the poset $\N$ (resp.~$\sN$); $\GP, \GQ, \GR$ (resp.~$\GM, \GN, \GL$) for general positive (resp.~negative) elements, i.e.~elements in the poset $\gP$ (resp.~$\gN$). The letters $A,B,C$ are used whenever we do not need to specify the poset. 

An {\em order-type} over $n\in \mathbb{N}$ is an $n$-tuple $\epsilon\in \{1, \partial\}^n$.  
%For every order type $\epsilon$, we denote its {\em opposite} order type by $\epsilon^\partial$, that is, $\epsilon^\partial_i = 1$ iff $\epsilon_i=\partial$ for every $1 \leq i \leq n$. 
For any order type $\varepsilon$, we let $\bbA^\varepsilon: = \Pi_{i = 1}^n \bbA^{\varepsilon_i}$. 
We use $n_h \in \mathbb{N}$ to denote the arity of a connective $h$. The language $\mathcal{L}_\mathrm{\FPLG}(\mathcal{F}, \mathcal{G})$ (from now on abbreviated as $\mathcal{L}_\mathrm{\FPLG}$) takes as parameters: two disjoint denumerable sets of proposition letters $\mathsf{AtProp^+}$, elements of which are denoted $p,q$, and $\mathsf{AtProp^-}$, elements of which are denoted $m,n$, and two disjoint sets of connectives:  
\begin{equation}\label{eq:order-types}
\begin{array}{c}
    \F = \{\mand, \mandl, \mandr, 
    \mdlarr, \mdrarrl, \mdlarrr, 
    \mdrarr, \mdrarrl, \mdrarrr,
    \sbuar, \buar\} \\
    
    \G = \{\mor, \morl, \morr, 
    \mrarr, \mrarrl, \mrarrr, 
    \mlarr, \mlarrl, \mlarrr,
    \srdar, \rdar\} \\[2mm]

     \epsilon(\mand) = \epsilon(\mandl) = \epsilon(\mandr) = 
     \epsilon(\mor) = \epsilon(\morl) = \epsilon(\morr) = (1,1) \\
     
     \epsilon(\mdrarr) = \epsilon(\mdrarrl) = \epsilon(\mdrarrr) = 
     \epsilon(\mrarr) = \epsilon(\mrarrl) = \epsilon(\mrarrr) = (\partial,1) \\
     
     \epsilon(\mdlarr) = \epsilon(\mdlarrl) = \epsilon(\mdlarrr) = 
     \epsilon(\mlarr) = \epsilon(\mlarrl) = \epsilon(\mlarrr) = (1,\partial) \\
     
     \epsilon(\srdar) = \epsilon(\rdar) = \epsilon(\sbuar) = \epsilon(\buar) = (1)
\end{array}
\end{equation}

\begin{definition}\label{def:fplg}
A fully polarized LG-algebra ($\FPLG$) $\bbA$ is defined by four posets $(\P, \arvd)$, $(\sP, \arvvvd)$, $(\N, \abvd)$ and $(\sN, \abvvvd)$ together with
\begin{itemize}
    \item Two adjunctions $\sbuar \dashv \rdar$ and $\buar \dashv \srdar$
\begin{equation}\label{eq:shifts-fplg}
\begin{pic}
\draw (0,0) node(p){$\P$};\draw (50,0) node(sn){$\sN$};
\draw (100,0) node(sp){$\sP$};\draw (150,0) node(n){$\N$};
\draw[->] (p) edge[bend left] node[above]{$\sbuar$} (sn);
\draw[->] (sp) edge[bend left] node[above]{$\buar$} (n);
\draw[->] (n) edge[bend left] node[below]{$\srdar$} (sp);
\draw[->] (sn) edge[bend left] node[below]{$\rdar$} (p);
\draw (25,0) node{\rotatebox{90}{$\vdash$}};
\draw (125,0) node{\rotatebox{90}{$\vdash$}};
\end{pic}
\end{equation}
    We use $\vvdh$ for the weakening relation represented by $\sbuar \dashv \rdar$ and $\hvvd$ for the weakening relation represented by $\buar \dashv \srdar$.

    \item Three weakening relations $\hrvvd\, \subseteq \P \times \sP$, $\hvd\, \subseteq \P \times \N$ and $\hbvvd\, \subseteq \sN \times \N$ such that for all $P \in \P$ and $N \in \N$ we have
    \begin{equation}\label{eq:shift-intro-elim}
        \sbuar P \hbvvd N ~~\text{ iff }~~ P \hvd N 
        ~~\text{ iff }~~ P \hrvvd \srdar N
    \end{equation}
    i.e. $\hvd$ is the weakening relation represented by the heterogeneous adjunction $\sbuar \dashv_{\hbvvd}^{\textcolor{red}{\fns\raisebox{\depth}{\scalebox{1}[-1]{$\preceqq$}}}} \srdar$.

    We define the collage posets $(\gP, \arvD) = (\P \sqcup \sP, \arvd \sqcup \hrvvd \sqcup \arvvvd)$, $(\gN, \abvD) = (\N \sqcup \sN, \abvd \sqcup \hbvvd \sqcup \abvvvd)$ and the collage weakening relation $\hvD = \vvdh \sqcup \hvd \sqcup \hvvd \ \subseteq \gP \times \gN$, summarised in Fig.~\ref{fig:fplg-wr}.

\begin{figure}[ht!]
    \centering
\begin{pic}\tikzset{outer sep = 0.5em}
\draw (0,0) node(p){$\P$};\draw (50,0) node(sn){$\sN$};
\draw (0,50) node(sp){$\sP$};\draw (50,50) node(n){$\N$};
\draw[->] (p) edge[pro] node[left]{$\hrvvd$} (sp)
    edge[loop left] node[left]{$\arvd$} (p)
    edge[pro] node[above left]{$\hvd$} (n)
    edge[pro] node[below]{$\vvdh$} (sn);
\draw[->] (sn) edge[pro] node[right]{$\hbvvd$} (n)
    edge[loop right] node[right]{$\abvvvd$} (sn);
\draw[->] (sp) edge[pro] node[above]{$\hvvd$} (n)
    edge[loop left] node[left]{$\arvvvd$} (sp);
\draw[->] (n) edge[loop right] node[right]{$\abvd$} (n);
\draw (-20,-20) edge[dashed] (70,-20);
\draw (-23.5,3) -- (-18,0); \draw (74.5,-3) -- (69,0);
\draw (-23.5,53) -- (-18,50); \draw (74.5,47) -- (69,50);
\draw (0,-40) node(gp){$\gP$};\draw (50,-40) node(gn){$\gN$};
\draw[->] (gp) edge[pro] node[above]{$\hvD$} (gn)
    edge[loop left] node[left]{$\arvD$} (gp);
\draw[->] (gn) edge[loop right] node[right]{$\abvD$} (gn);
\draw (-23.5,-37) -- (-18,-40); \draw (74.5,-43) -- (69,-40);
\end{pic}
    \caption{\ \ Weakening relations in $\FPLG$-algebras.}
    \label{fig:fplg-wr}
\end{figure}

    \item Six operations (that we call LG-connectives)
\begin{equation}\notag
\begin{array}{rlrlrl}
     \mand:& \gP \times \gP \to \P &
     \mdlarr:& \gP \times \gN\Pa \to \P&
     \mdrarr:& \gN\Pa \times \gP \to \P\\
     
     \mor:& \gN \times \gN \to \N&
     \mrarr:& \gP\Pa \times \gN \to \N&
     \mlarr:&  \gN \times \gP\Pa \to \N
\end{array}
\end{equation}
    
    such that the following heterogeneous adjunctions hold
\begin{equation}\label{eq:adj-fplg}
\begin{array}{cc}
     \GQ \hvD \GP \mrarr \GN ~~~~\text{iff}~~~~
     \GP \mand \GQ \hvD \GN ~~~~\text{iff}~~~~
     \GP \hvD \GN \mlarr \GQ 
     
     \\
     
     \GP \mdlarr \GN \hvD \GM ~~~~\text{iff}~~~~
     \GP \hvD \GM \mor \GN ~~~~\text{iff}~~~~
     \GM \mdrarr \GP \hvD \GN
\end{array}
\end{equation}

    \item Finally, 12 operations (that we call $\ell$-variants and $r$-variants, or simply LG-variants)
\begin{equation}\notag
\begin{array}{rlrlrl}
     \mandl:& \gN \times \gP \to \sN &
     \mdlarrl:& \gN \times \gN\Pa \to \sN&
     \mdrarrl:& \gP\Pa \times \gP \to \sN\\
     
     \morl:& \gP \times \gN \to \sP&
     \mrarrl:& \gN\Pa \times \gN \to \sP&
     \mlarrl:&  \gP \times \gP\Pa \to \sP\\
     
     \mandr:& \gP \times \gN \to \sN &
     \mdlarrr:& \gP \times \gP\Pa \to \sN&
     \mdrarrr:& \gN\Pa \times \gN \to \sN\\
     
     \morr:& \gN \times \gP \to \sP&
     \mrarrr:& \gP\Pa \times \gP \to \sP&
     \mlarrr:&  \gN \times \gN\Pa \to \sP
\end{array}
\end{equation}
    
    such that the following adjunctions hold
\begin{equation}\label{eq:adj-fplg-variants}
\arraycolsep=1.5mm
\begin{array}{cccccccccc}
     \GQ \arvD \GP \mrarrr \GR &\text{iff}&
     \GP \mand \GQ \arvD \GR &\text{iff}&
     \GP \arvD \GR \mlarrl \GQ, &
     
     \GL \mdlarrl \GN \abvD \GM &\text{iff}&
     \GL \abvD \GM \mor \GN &\text{iff}&
     \GM \mdrarrr \GL \abvD \GN\\

     \GQ \arvD \GL \mrarrl \GN &\text{iff}&
     \GL \mandl \GQ \abvD \GN &\text{iff}&
     \GL \abvD \GN \mlarr \GQ, &
     
     \GP \mdlarrr \GR \abvD \GM &\text{iff}&
     \GP \arvD \GM \morr \GR &\text{iff}&
     \GM \mdrarr \GP \arvD \GR\\

     \GQ \abvD \GP \mrarr \GN &\text{iff}&
     \GP \mandr \GL \abvD \GN &\text{iff}&
     \GP \arvD \GN \mlarrr \GL, &
     
     \GP \mdlarr \GN \arvD \GR &\text{iff}&
     \GP \arvD \GR \morl \GN &\text{iff}&
     \GR \mdrarrl \GP \abvD \GN
\end{array}
\end{equation}
\end{itemize}
\end{definition}

\begin{proposition}
\label{prop:WeakeningOnCollagePoset}
In any $\FPLG$ we have ${\hrvvd \! \hvvd} \ = \ \ \hvd \ \ = \ {\vvdh \! \hbvvd}$.
\end{proposition}

\begin{proof}\label{prop:wr-equalities}
We show that ${\hrvvd \! \hvvd}  = \ \hvd$. Fix $P \in \P$, $N \in \N$ and assume that $P \hrvvd \SQ$ and $\SQ \hvvd N$ for some $\SQ \in \sP$. From $\SQ \hvvd N$, we conclude that $\SQ \arvvvd \srdar N$, for $\hvvd$ is the weakening relation represented by $\buar \dashv \srdar$ (see proposition \ref{prop:wr}). From $P \hrvvd \SQ$ and $\SQ \arvvvd \srdar N$ we conclude $P \hrvvd \srdar N$, for $\hrvvd$ is a weakening relation compatible with the partial order $\arvvvd$. Therefore, $P \hvd N$ by \eqref{eq:shift-intro-elim}.

Now, fix $P \in \P$, $N \in \N$ and assume $P \hvd N$. On the one hand, $P \hrvvd \srdar N$ by \eqref{eq:shift-intro-elim}. On the other hand, $\srdar N \arvvvd \srdar N$ gives $\srdar N \hvvd N$ by proposition \ref{prop:wr} on $\hvvd$. The equality $\hvd \ = {\vvdh \! \hbvvd}$ is proven in a similar way.
\end{proof}

\begin{remark}
 In \cite{Kurz--Moshier--Jung:2019}, operations that are order-reversing in some coordinate are considered `problematic', essentially because source and target of weakening relations are, in general, of different types. Remark 2.23 in \cite{Kurz--Moshier--Jung:2019} illustrates the concern considering negation and implication in Boolean or Heyting algebras as prototypical examples. We observe that this is an issue only insofar we confine ourselves to homogeneous operations. In the present setting, the problem is overcome allowing heterogeneous operations. In the case of fully polarized LG-algebras, shifts, the adjoints of shifts, LG-connectives, and LG-variants are heterogeneous operations (see definition \ref{def:fplg}).
\end{remark}

\section{Proof theory}\label{sec:calculus}

The basic Lambek-Grishin logic can be presented as a (single-type) proper display sequent calculus (see \cite{Moortgat:2009}). Section 2.1 of \cite{Moortgat--Moot:2011} provides a display sequent calculus for the basic Lambek-Grishin logic and its expansion with Grishin's \cite{Grishin:1993} `linear distributivity' structural rules capturing interaction between the $\otimes$ and the $\oplus$ families of connectives. Section 3.1 of \cite{Moortgat--Moot:2011} provides a focused sequent calculus \fLG for the same logic. In \cite{Moortgat--Moot:2011} the calculus \fLG is considered a display calculus given that all the connectives in the language are residuated in each coordinate, even tough the so-called \emph{display postulates} capturing residuation can only be applied in neutral phases. On the contrary, all the connectives of the calculus \fDLG introduced in section \ref{subsec:fDLG} are residuated and display postulates can be applied \emph{in any phase}. Therefore, \fDLG is a display calculus accordingly to the usual definition. It is worth to mention here that \fLG and \fDLG are equivalent calculi, indeed, it is not difficult to define faithful translations from \fLG-derivations to \fDLG-derivations and vice versa (see Section \ref{sec:ProofTranslations}). 

More in general, \fDLG has the following distinctive features: (i) homogeneous as well as heterogeneous connectives are considered (multi-type), (ii) each rule is closed under uniform substitution within each type (properness), (iii) every structure occurring in a derivable sequent can be isolated either in precedent or, exclusively, in succedent position by means of display postulates (display property), and (iv) homogeneous as well as heterogeneous turnstiles are considered. Any \emph{multi-type proper display calculus} (see \cite{Wansing:2002,FriGreKurPalSik16}) has features i-iii but not iv.

%%%%%
\subsection{Focused display LG-calculus}
\label{subsec:fDLG}

The language of \fDLG is the Lambek-Grishin display calculus language expanded with the structural $\ell,r$-variants and shifts operators.

\begin{notation}
Following the display calculi literature \cite{GJLPT-LE-logics}, we adopt a notation where structural and operational (aka logical) connectives are in a one-to-one correspondence. Moreover, we mark the structural counterpart of a connective $\star$ as follows: $\hat{\star}$ if $\star$ is a left-adjoint or residual, $\check{\star}$ if $\star$ is a right-adjoint or residual.
\end{notation}
%and $\tilde{\star}$ if $\star$ is both a left- and a right-adjoint.

The Lambek-Grishin structural and operational connectives are the following

\begin{center}\tabcolsep=0.7mm
\begin{tabular}{|r|c|c|c|c|c|c|c|c|c|c|c|c|c|c|c|c|c|c|}
\hline
\mc{1}{|c|}{\scriptsize{Structural symbols}} &
\mc{1}{|c|}{$\rule{0pt}{2.25ex}\MAND$} &
$\MDLARR$ &
\mc{1}{c|}{$\MDRARR$} &
\mc{1}{|c|}{$\MOR$} &
$\MRARR$ &
\mc{1}{c|}{$\MLARR$}\\
\hline
\mc{1}{|c|}{\scriptsize{Operational symbols}} &
\mc{1}{|c|}{\rule{0pt}{2.25ex}$\mand$} &
$\mdlarr$ &
\mc{1}{c|}{$\mdrarr$} &
\mc{1}{|c|}{$\mor$} &
$\mrarr$ &
\mc{1}{c|}{$\mlarr$}\\
\hline
\end{tabular}
\end{center}

\noindent Below we list  the structural $\ell,r$-variants included in the language of \fDLG and the corresponding operational $\ell,r$-variants (in grey cells). We consider $\ell,r$-variants essentially to ensure the display property, therefore we find more convenient to not include at all the operational $\ell,r$-variants in the language of \fDLG. 
The subscript $\ell$ (resp. $r$) of a $\ell$-variant $\star_{\ell}$ (resp. $r$-variant $\star_r$) indicates that the subformula on its left is of the opposite polarity w.r.t. the corresponding LG-connective $\star$.

\begin{center}\tabcolsep=0.7mm
\begin{tabular}{|r|c|c|c|c|c|c|c|c|c|c|c|c|c|c|c|c|c|c|}
\hline
\mc{1}{|c|}{\scriptsize{Structural symbols}} &
\mc{1}{|c|}{$\rule{0pt}{2.25ex}\MANDL$} &
$\MDLARRL$ &
\mc{1}{c|}{$\MDRARRL$} &
\mc{1}{|c|}{$\MORL$} &
$\MRARRL$ &
\mc{1}{c|}{$\MLARRL$} &
\mc{1}{|c|}{$\MANDR$} &
$\MDLARR_r$ &
\mc{1}{c|}{$\MDRARRR$} &
\mc{1}{|c|}{$\MORR$} &
$\MRARR_r$ &
\mc{1}{c|}{$\MLARRR$} \\
\hline
\mc{1}{|c|}{\scriptsize{Operational symbols}} &
\mc{1}{|c|}{\rule{0pt}{2.25ex}\graycell$\mandl$} &
\graycell$\mdlarrl$ &
\mc{1}{c|}{\graycell$\mdrarrl$} &
\mc{1}{|c|}{\graycell$\morl$} &
\graycell$\mrarrl$ &
\mc{1}{c|}{\graycell$\mlarrl$} &
\mc{1}{|c|}{\graycell$\mandr$} &
\graycell$\mdlarr_r$ &
\mc{1}{c|}{\graycell$\mdrarrr$} &
\mc{1}{|c|}{\graycell$\morr$} &
\graycell$\mrarrr$ &
\mc{1}{c|}{\graycell$\mlarrr$}\\
\hline
\end{tabular}
\end{center}

\noindent Below we list the structural and operational shifts operators. We find more convenient to not include the operational adjoints of shifts (in grey cells) in the language of \fDLG

\begin{center}
\begin{tabular}{|r|c|c|c|c|c|c|c|c|}
\hline
\mc{1}{|c|}{\scriptsize{Structural symbols}} &
\mc{1}{|c|}{$\rule{0pt}{2.25ex}\RDARR$} &
\mc{1}{c|}{$\SRDARR$} & 
\mc{1}{|c|}{$\BUARR$} &
\mc{1}{c|}{$\SBUARR$} \\
\hline
\mc{1}{|c|}{\scriptsize{Operational symbols}} &
\mc{1}{|c|}{\graycell$\rdar$} &
\mc{1}{c|}{$\srdar$} & 
\mc{1}{|c|}{\graycell$\buar$} &
\mc{1}{c|}{$\sbuar$} \\
\hline
\end{tabular}
\end{center}

For any connective $h$ (either structural or operational), the arity $n_h$, order-type $\epsilon(h)$, and its classification as $\mathcal{F}$-connective or , exclusively, $\mathcal{G}$-connective are like in \ref{eq:order-types}. 

\begin{notation}
We adopt the following notational convention for formulas and structures: $\GP\in\{P,\SP\}, \GX\in\{X,\SX\}, \GN\in\{N,\SN\}, \GD\in\{\Delta,\SD\}$. For instance, accordingly to this convention, we have that $\GP \mand \GQ \in \{P \mand Q, \SP \mand Q, P \mand \SQ, \SP \mand \SQ\}$. Therefore, general formulas and structures are not a full-fledged sort, but rather an abbreviation.
\end{notation}

The calculus \textbf{fD.LG} manipulates formulas and structures defined by the following mutual recursion, where $p \in \mathsf{AtProp}^+$ and $n \in \mathsf{AtProp}^-$:
\begin{equation}
\notag
\tabcolsep=1.7mm
\begin{tabular}{r@{}rcll}
$\mathsf{PurePosFm} \ni\ $ & $P$ & $::=$ & $p \sep \GP \mand \GP \sep \GP \mdlarr \GN \sep \GN \mdrarr \GP$ & \text{Pure positive formulas} \\

$\mathsf{PureNegFm} \ni\ $ & $N$ & $::=$ & $n \sep \GN \mor \GN \sep \GP \mrarr \GN \sep \GN \mlarr \GP$ & \text{Pure negative formulas} \\

$\mathsf{ShiftedPosFm} \ni\ $ & $\SP$ & $::=$ & $\srdar\, N$ & \text{Shifted positive formulas} \\
     
$\mathsf{ShiftedNegFm} \ni\ $ & $\SN$ & $::=$ & $\sbuar P$ & \text{Shifted negative formulas} \\

% & & \\

$\mathsf{GenPosFm} \ni\ $ & $\GP$ & $::=$ & $P \sep\SP$ & \text{General positive formulas} \\

$\mathsf{GenNegFm} \ni\ $ & $\GN$ & $::=$ & $N \sep\SN$ & \text{General negative formulas} \\

% & & \\

$\mathsf{PurPosStr} \ni\ $ & $X$ & $::=$ & $P \sep\RDARR\, \SD \sep
\GX \MAND \GX \sep \GX \MDLARR \GD \sep \GD \MDRARR \GX$ & \text{Pure positive structures}\\

$\mathsf{PurNegStr} \ni\ $ & $\Delta$ & $::=$ & $N \sep\BUARR\, \SX \sep\GD \MOR \GD \sep \GX \MRARR \GD \sep\GD \MLARR \GX$ & \text{Pure negative structures} \\

$\mathsf{ShiftedPosStr} \ni\ $ & $\SX$ & $::=$ & $\SP \sep\SRDARR\, \Delta \sep \GX \MORL \GD \sep \GD \MORR \GX \sep$ & \text{Shifted positive structures} \\
& & & $\GD \MRARRL \GD \sep\GX \MRARRR \GX \sep \GX \MLARRL \GX \sep\GD \MLARRR \GD$ & \\

$\mathsf{ShiftedNegStr} \ni\ $ & $\SD$ & $::=$ & $\SN \sep\SBUARR\, X \sep\GD \MANDL \GX \sep\GX \MANDR \GD \sep$ & \text{Shifted negative structures} \\
& & & $\GD \MDLARRL \GD \sep\GX \MDLARRR \GX \sep \GX \MDRARRL \GX \sep\GD \MDRARRR \GD$ & \\

% & & \\

$\mathsf{GenPosStr} \ni\ $ & $\GX$ & $::=$ & $X \sep\SX$ & \text{General positive structures} \\

$\mathsf{GenNegStr} \ni\ $ & $\GD$ & $::=$ & $\Delta \sep\SD$ & \text{General negative structures} \\

\end{tabular}
\end{equation}

The well-formed sequents are the following:
\begin{equation}\label{eq:turnstiles}
\text{\begin{tabular}{|r|c|c|c|c|}\hline
\mc{1}{|c|}{\scriptsize{Positive sequents}} &
\mc{1}{|c|}{$X \rvd Y$} &
$\SX \urvvd Y$ &
\mc{1}{|c|}{$X \rvvd \SY$} &
$\SX \rvvvd \SY$ \\
\hline
\mc{1}{|c|}{\scriptsize{Negative sequents}} &
\mc{1}{|c|}{$\Delta \bvd \Gamma$} &
$\SD \bvvd \Gamma$ & 
\mc{1}{|c|}{$\Delta \ubvvd \SG$} &
$\SD \bvvvd \SG$ \\
\hline
\mc{1}{|c|}{\scriptsize{Neutral sequents}} &
\mc{1}{|c|}{$X \vd \Delta$} &
$\SX \tvvd \Delta$ &
\mc{1}{|c|}{$X \vvdt \SD$} &
$\SX \vvvd \SD$ \\
\hline
\end{tabular}
}
\end{equation}

\begin{notation}
We extend the previous conventions to sequents as follows: 
$\nrvD \in\{\nrvd, \urvvd, \nrvvd, \nrvvvd\}$, $\nbvD \in\{\nbvd, \nbvvd, \ubvvd, \nbvvvd\}$, $\nvD \in\{\nvd, \nvvd, \vvdt, \nvvvd\}$. The reading is supposed to preserve well-formedness. For instance, in a premise of a binary logical rule $\nrvD = \nrvd$ iff $\GX = X$ and $\GY = Y$, or $\nrvD = \nrvvd$ iff $\GX = X$ and $\GY = \SY$, and so on. Therefore, each binary logical rule below denotes four different rules. Nonetheless, notice that in any actual derivation the instantiation of a logical inference rule is unique and completely deterministic.
\end{notation}

%%%
The calculus \fDLG consists of the following rules.
\paragraph*{Axioms and cuts}
{\small
\begin{equation}
\label{eq:(co)axioms-and-cut-rules}
\begin{tabular}{rcl}
\AXC{$\ $}
\RL{\fns $p$-Id}
\UIC{$p \rvd  p$}
\DP
 & &
\AXC{$\ $}
\LL{\fns $n$-Id}
\UIC{$n \bvd n$}
\DP
\\
 & & \\
\AXC{$\GX \rvD \GP$}
\AXC{$\GP \rvD \GY$}
\LL{\fns P-Cut}
\BIC{$\GX \rvD \GY$}
\DP 
 & &
\AXC{$\GG \bvD \GN$}
\AXC{$\GN \bvD \GD$}
\RL{\fns N-Cut}
\BIC{$\GG \bvD \GD$}
\DP 
\\
 & & \\
\AXC{$\GX \rvD \GP$}
\AXC{$\GP \vD \GD$}
\LL{\fns Pn-Cut}
\BIC{$\GX \vD \GD$}
\DP 
 & &
\AXC{$\GX \vD \GN$}
\AXC{$\GN \bvD \GD$}
\RL{\fns nN-Cut}
\BIC{$\GX \vD \GD$}
\DP
\end{tabular}
\end{equation}
}
\paragraph*{Logical rules}
The logical rules transforming a structural connective in the premise into its logical counterpart in the conclusion are called \emph{translation rules}. All the other logical rules are called \emph{tonicity rules}. In the literature on focused calculi, `asynchronous' and `synchronous', respectively, are often used (e.g.~in \cite{Andreoli:2001}).
{\small
\begin{equation}\label{eq:operational-rules}
\begin{tabular}{rl}

\AXC{$\GP \MAND \GQ  \vD \GD$}
\LL{\fns $\mand_L$}
\UIC{$\GP \mand \GQ \vD \GD$}
\DP
 \ 
\AXC{$\GX \rvD \GP$}
\AXC{$\GY \rvD \GQ$}
\RL{\fns $\mand_R$}
\BIC{$\GX \MAND \GY \rvd \GP \mand \GQ$}
\DP

 &

\AXC{$\GN \bvD \GG$}
\AXC{$\GM \bvD \GD$}
\LL{\fns $\mor_L$}
\BIC{$\GN \mor \GM \bvd \GG \MOR \GD$}
\DP
 \ 
\AXC{$\GX \vD \GN \MOR \GM$}
\RL{\fns $\mor_R$}
\UIC{$\GX \vD \GN \mor \GM$}
\DP
 \\

 & \\

\AXC{$\GP \MDLARR \GN \vD \GD$}
\LL{\fns $\mdlarr_L$}
\UIC{$\GP \mdlarr \GN \vD \GD$}
\DP
 \ 
\AXC{$\GX \rvD  \GP$}
\AXC{$\GN \bvD \GD$}
\RL{\fns $\mdlarr_R$}
\BIC{$\GX \MDLARR \GD \rvd \GP \mdlarr \GN$}
\DP

 & 
 
\AXC{$\GX \rvD  \GP$}
\AXC{$\GN \bvD \GD$}
\LL{\fns $\mrarr_L$}
\BIC{$\GP \mrarr \GN \bvd X \MRARR \Delta$}
\DP
 \ 
\AXC{$\GX \vD \GP \MRARR \GN$}
\RL{\fns $\mrarr_R$}
\UIC{$\GX \vD \GP \mrarr \GN$}
\DP
\\

 & \\

\AXC{$\GN \MDRARR \GP \vD \GD$}
\LL{\fns $\mdrarr_L$}
\UIC{$\GN \mdrarr \GP \vD \GD$}
\DP
 \ 
\AXC{$\GN \bvD \GD$}
\AXC{$\GX \rvD  \GP$}
\RL{\fns $\mdrarr_R$}
\BIC{$\GD \MDRARR \GX \rvd \GN \mdrarr \GP$}
\DP

 & 

\AXC{$\GN \bvD \GD$}
\AXC{$\GX \rvD  \GP$}
\LL{\fns $\mlarr_L$}
\BIC{$\GN \mlarr \GP \bvd \GD \MLARR \GX$}
\DP
 \ 
\AXC{$\GX \vD \GN \MLARR \GP$}
\RL{\fns $\mlarr_R$}
\UIC{$\GX \vD \GN \mlarr \GP$}
\DP
\\

& \\

\AXC{$N  \bvd \Delta$}
\LL{\fns $\srdar_L$}
\UIC{$\srdar N \rvvvd \SRDARR\, \Delta$}
\DP
 \
\AXC{$\GX \rvD \SRDARR\, N$}
\RL{\fns $\srdar_R$}
\UIC{$\GX \rvD \srdar N$}
\DP

 & 

\AXC{$\SBUARR\, P \bvD \GD$}
\LL{\fns $\sbuar_L$}
\UIC{$\sbuar P \bvD \GD $}
\DP
 \ 
\AXC{$X  \rvd P$}
\RL{\fns $\sbuar_R$}
\UIC{$\SBUARR X  \bvvvd \sbuar P$}
\DP
% \\
\end{tabular}
\end{equation}
}
\paragraph*{Display postulates}
Below we use a double inference line to denote two rules: (i) from the premise to the conclusion and (ii) from the conclusion to the premise. We use the same name for both rules.
{\small
\begin{equation}
\label{eq:display-postulates}
\begin{tabular}{cc}
\AXC{$\GY \vD \GX \MRARR \GD$}
\LL{\fns $\hat{\mand} \dashv \check{\mrarr}$}
\doubleLine
\UIC{$\GX \MAND \GY \vD \GD$}
\LL{\fns $\hat{\mand} \dashv \check{\mlarr}$}
\doubleLine
\UIC{$\GX \vD \GD \MLARR \GY$}
\DP
 \ 
\AXC{$\GY \rvD \GX \MRARR_r \GZ$}
\RL{\fns $\hat{\mand} \dashv \check{\mrarr}_r$}
\doubleLine
\UIC{$\GX \MAND \GY \rvD \GZ$}
\RL{\fns $\hat{\mand} \dashv \check{\mrarr}_\ell$}
\doubleLine
\UIC{$\GX \rvD \GZ \MLARRL \GY$}
\DP

 & 
\AXC{$\GSi \MDLARRL \GD \bvD \GG$}
\LL{\fns $\MDLARRL \dashv \MOR$}
\doubleLine
\UIC{$\GSi \bvD \GG \MOR \GD$}
\LL{\fns $\MDRARRR \dashv \MOR$}
\doubleLine
\UIC{$\GG \MDRARRR \GSi \bvD \GD$}
\DP
 \ 
\AXC{$\GX \MDLARR \GD \vD \GG$}
\RL{\fns $\MDLARR \dashv \MOR$}
\doubleLine
\UIC{$\GX \vD \GG \MOR \GD$}
\RL{\fns $\MDRARR \dashv \MOR$}
\doubleLine
\UIC{$\GG \MDRARR \GX \vD \GD$}
\DP

 \\

 &  \\

\AXC{$\GG \MDRARR \GX \rvD \GY$}
\doubleLine
\LL{\fns $\MDRARR \dashv \MORR$}
\UIC{$\GX \rvD \GG \MORR \GY$}
\doubleLine
\LL{\fns $\MDLARRR\! \dashv \MORR$}
\UIC{$\GX \MDLARRR \GY \bvD \GG$}
\DP
 \ 
\AXC{$\GX \MDLARR \GD \rvD \GY$}
\doubleLine
\RL{\fns $\MDLARR \dashv \MORL$}
\UIC{$\GX \rvD \GY \MORL \GD$}
\doubleLine
\RL{\fns $\MDRARRL\! \dashv \MORL$}
\UIC{$\GY \MDRARRL \GX \bvD \GD$}
\DP

 & 

\AXC{$\GG \bvD \GX \MRARR \GD$}
\doubleLine
\LL{\fns $\MANDR \dashv \MRARR$}
\UIC{$\GX \MANDR \GG \bvD \GD$}
\doubleLine
\LL{\fns $\MANDR\! \dashv \MLARRR$}
\UIC{$\GX \rvD \GD \MLARRR \GG$}
\DP
 \ 
\AXC{$\GY \rvD \GG \MRARRL \GD$}
\RL{\fns $\MANDL \dashv \MRARRL$}
\doubleLine
\UIC{$\GG \MANDL \GY \bvD \GD$}
\doubleLine
\RL{\fns $\MANDL \dashv \MLARR$}
\UIC{$\GG \bvD \GD \MLARR \GY$}\DP
 
\\ 

\\

\mc{2}{c}{
\AXC{$\SBUARR\, X \bvvvd \SD$}
\doubleLine
\RL{\fns $\SBUARR \dashv \RDARR$}
\UIC{$X \rvd \RDARR\, \SD$}
\DP
 \ \ \ \ \ 

\AXC{$\SBUARR X \bvvd \Delta$}
\doubleLine
\RL{$\SBUARR \dashv \SRDARR$}
\UIC{$X \rvvd \SRDARR \Delta$}
\DP
 \ \ \ \ 

 \
\AXC{$\SX \rvvvd \SRDARR\, \Delta$}
\doubleLine
\LL{\fns $\BUARR \dashv \SRDARR$}
\UIC{$\BUARR\, \SX \bvd \Delta$}
\DP
}

% \\
\end{tabular}
\end{equation}
}
\paragraph*{Structural rules}
{\small
\begin{equation}
\begin{tabular}{cc}
\AXC{$\GX \vD  \Delta$}
\RL{\fns $\SRDARR$}
\doubleLine
\UIC{$\GX \rvD \SRDARR\, \Delta$}
\DP
 &
\AXC{$X \vD \GD$}
\LL{\fns $\SBUARR$}
\doubleLine
\UIC{$\SBUARR\, X \bvD \GD$}
\DP
\\
\end{tabular}
\end{equation}
}

%Structural shifts $\SRDARR$ and $\SBUARR$ control the entry and exit of black sequents, but stay in neutral phase (i.e. preserves neutral sequents) by definition \ref{def:focused-sequent}. An example derivation is given in Fig.~\ref{fig:example-derivation}.
% Their structural rules $\SRDARR$ and $\SBUARR$ together with the other structural rules on black sequents account for commutation between shifts and 

%We call $\T$ the set of turnstiles and we use $\Phi, \Psi$ for structure meta-variables of any polarity (i.e. positive or negative). Following the tradition of algebraic proof-theory \cite{GJLPT-LE-logics}, we note $\F$ the set of left adjoint connectives (i.e. pure positive or shifted negative) and $\G$ for right adjoints (i.e. pure negative or shifted positive).

\begin{proposition}
\label{prop:NonDerivableSequents}
Sequents of the form $\SX \urvvd Y$, $\Delta \ubvvd \SG$ and $\SX \vvvd \SD$ are not derivable.
\end{proposition}

\begin{proof}
By quick induction on the derivation and examination of every rule with the following induction hypothesis: ``In a sequent $S$, if $\SX$ (resp. $\SD$) occurs in $S$ in precedent (resp. succedent) position, then put in display, the succedent (resp. precedent) is either pure negative or shifted positive (resp. pure positive or shifted negative).'' It namely works because LG connectives are pure, $\ell,r$-variants are shifted and because it holds for the conclusion of rules involving shifts.
\end{proof}

Indeed, some conceivable combinations of cut rules are actually not included in the calculus (see \eqref{eq:(co)axioms-and-cut-rules}), as well as some conceivable weakening relations are not considered in the algebraic semantics (see definition \ref{def:fplg}).

\subsection{Focalization}

In this subsection we first provide a procedural description and a formal definition of \emph{strongly focused proof} of an arbitrary sequent calculus (definition \ref{def:strong-foc}, adapted from \cite[def. 3]{Laurent:2004}). Then, we show that \fDLG has strong focalization (theorem \ref{thm:cfdlg-strongly-focalizaed}). In the end we provide some nomenclature and a diagrammatic representation of the `topology of rules' of \fDLG. We use $\Psi, \Phi$ to refer to arbitrary structures.

%%%%%
%\paragraph*{Focalization and display calculi}

%We give here a formal definition for focalization (weak focalization and strong focalization) and prove that every cut-free and variant-free proof of \fDLG is strongly focalized (theorem \ref{thm:cfdlg-strongly-focalizaed}). We end this section by a important remark about the semantic footprint of focalization and phases.
%Focalization is the property that tonicity rules constructing a formula $A$ shall not intertwine with any other rule. 

The backward-looking proof search strategy implemented by a focused sequent calculus (see for instance \cite{Andreoli:2001}) can be roughly described as follows: (i) pick a formula, (ii) decompose the chosen formula as much as possible via applications of non-invertible logical rules, (iii) once you reach a subformula of the opposite polarity or an atom, then you may apply structural rules or invertible logical rules, (iv) repeat the process. In order to make precise this informal procedural description, we use a couple of preliminary definitions (see for instance \cite{GJLPT-LE-logics}). 
\begin{definition}[Signed generation tree]
The positive (resp. negative) generation tree of a structure $\Psi$, denoted $+\Psi$ (resp. $-\Psi$), is defined by labelling the root node of the generation tree of $\Psi$ with the sign $+$ (resp. $-$), and then propagating the labelling on each remaining node as follows:

\begin{itemize}
    \item[] For any node labelled with $h \in \F \cup \G$ of arity $n_h \geq 1$, and for any $1 \leq i \leq n_h$, assign the same (resp. the opposite) sign to its $i$-th child node if the order-type $\epsilon(h,i) = 1$ (resp. if $\epsilon(h,i) = \partial$).
\end{itemize}

The signed generation tree of a sequent $\Psi \t \Phi$ consists of the signed generation trees $+\Psi$ and $-\Phi$.
\end{definition}

\begin{definition}[Skeleton and PIA]
A node in a signed generation tree of a sequent is called skeleton if it is labelled with $+f$ for some $f \in \F$ or with $-g$ for some $g \in \G$. Otherwise, it is called a PIA node.
\end{definition}

An example of signed generation tree is given in Fig.~\ref{fig:ex-generation-tree} in appendix \ref{ap:examples}. Notice that any signed generation tree of a well-formed stucture $\Psi$ in the language of $\fDLG$ can be partitioned into skeleton vs PIA subtrees (i.e.~connected subgraphs of the signed generation tree of $\Psi$).

\begin{definition}[Transition node]\label{def:transition-node}
A transition node of a signed generation tree $\sigma$ is the uppermost node of a skeleton or PIA subtree excluding the root of $\sigma$.
\end{definition}

\begin{definition}[Proof-section]
\label{def:proof-section}
% An \emph{immediate proof-section} of a proof $\pi$ is obtained by pruning $\pi$, i.e.~by removing either the leaves or the root of $\pi$. An arbitrary \emph{proof-section} of a proof $\pi$ is obtained by recursive applications of pruning. Therefore, a proof-section is a proof where the leaves are derivable sequents, but not necessarily axioms.
A proof-section $\pi'$ of a proof-tree $\pi$ is a connected subgraph of $\pi$, such that for every node $S \in \pi'$, if $S$ is not a leaf of $\pi'$ and it is introduced by a rule application $R$, then also the premise(s) of $R$ are in $\pi'$.
% and each child-node is introduced by a rule of the calculus.
\end{definition}

%\begin{definition}[Weak focalization]\label{def:weak-foc}
%A proof $\pi$ is strongly focalized if it is cut-free and, for every subproof $\pi'$ of $\pi$ with end sequent $S$ containing a formula $A$ which is a $\F$-formula in succedent position or a $\G$-formula in precedent position, \emph{any tonicity rule} between rules introducing connectives of $\tau(A)$ is a rule introducing a connective of $\tau(A)$.
%\end{definition}

\begin{definition}[Strong focalization]\label{def:strong-foc}
%A proof $\pi$ is strongly focalized if it is cut-free and, for every subproof $\pi'$ of $\pi$ with end sequent $S$ containing a formula $A$ which is a $\F$-formula in succedent position or a $\G$-formula in precedent position, \emph{any rule} between rules introducing connectives of $\tau(A)$ is a rule introducing a connective of $\tau(A)$.
%
%Let $\pi'$ be any sub-proof of a cut-free proof $\pi$ and $A$ a formula occurring in the end-sequent of $\pi'$. Let $\sigma$ be a subtree of the generation tree of $A$ such that every node in $\sigma$ is PIA. $\pi$ is strongly focalized if $\sigma$ is compositionally built in $\pi$ by means of a sequence of applications of tonicity rules only.
%

% Let $A$ be a formula occurring in a cut-free sequent proof $\pi$. The proof $\pi$ is \emph{strongly focalized} if every PIA subtree of $A$ is constructed by a proof-section of $\pi$ containing only tonicity rules.
A sequent proof $\pi$ is \emph{strongly focalized} if cut-free and, for every formula $A$ occurring in $\pi$, every PIA subtree of $A$ is constructed by a proof-section of $\pi$ containing only tonicity rules. 

%A cut-free sequent proof $\pi$ is \emph{strongly focalized} if for every formula $A$ occurring in a subproof of $\pi$ and every PIA subtree $\Sigma$ of $A$, $\pi$ admits a proof-section only consisting of the applications of tonicity rules of all the connectives of $\Sigma$.
\end{definition}

\begin{proposition}\label{prop:connective-introduction-generation-tree}
Let $h$ be an operational connective occurring in the generation tree of the end-sequent $\Psi \t \Phi$ in a $\fDLG$-proof $\pi$, and let $S$ be the uppermost sequent in $\pi$ where $h$ occurs. If $h$ is a skeleton node, then it is introduced in $S$ via a translation rule. If $h$ is a PIA node, then it is introduced in $S$ via a tonicity rule.
\end{proposition}
\begin{proof}
Immediate by inspection of the rules of $\fDLG$.
\end{proof}

%
%According to proposition \ref{prop:connective-introduction-generation-tree}, transition nodes (or more exactly the edges between a transition node and its parent) correspond to phase transition in a strongly focalized proof.
%
\begin{proposition}\label{prop:transition-node-shift}
% A node of a \cfDLG-proof is a transition node (see definition \ref{def:transition-node}) iff it introduces an exit or an entry node (see definition \ref{def:EntryExitPoints}), so it is labelled by $\srdar$ or $\sbuar$.
%, and shifts can label transition nodes or the root of a formula signed generation tree.
%(cfr.~definition \ref{def:EntryExitPoints})
Let $A$ be a formula occurring in a $\ell$-$r$-variant-free sequent. If a shift labels a node $\nu$ of the signed generation tree of $A$, then either $\nu$ is a transition node or it is the root of $A$.
\end{proposition}

\begin{proof}
This is due to the presence of shift operators and the polarization of the calculus: For every LG structural connective $\star \in \F$ (resp. $\star \in \G$), (i) the target sort of $\star$ is positive (resp. negative), and (ii) the source sort of the $i$-th argument of $\star$ is positive (resp. negative) iff $\epsilon(\star,i) = 1$.
\end{proof}

%\begin{proposition}
%For every formula $A$ of \fDLG, every skeleton or PIA subtree of $A$ contains at least one LG-connective or atom.
%\end{proposition}
%
%\begin{proof}
%This is due to the fullness of the polarization: the target sort of LG connectives and atoms is pure whereas the target sort of $\srdar$ and $\sbuar$ is shifted, so that $\srdar$ and $\sbuar$ cannot compose.
%\end{proof}
%
%This forbids phases to be empty of LG connective rules, thus ensuring strong focalization if no other rule is applicable on focused sequents.

%\begin{definition}\label{def:main-signed-tree}
%The \textbf{main tree} of a signed generation tree is the (set of connective occurrences of the) skeleton or PIA subtree containing the root of the tree.
%\end{definition}

%In \fDLG, the main tree $\tau(\Psi)$ of a structure $\Psi$ is equivalently defined by \eqref{eq:main-tree}.
%
%\begin{equation}\label{eq:main-tree}
%\begin{array}{cl}
%    \tau(\bullet(\Psi)) = \{\bullet\} \cup \tau'(\Psi) ~~~~~~~~
%    \tau'(\bullet(\Psi)) = \emptyset &
%    \text{for every } \bullet \in \{\srdar, \sbuar, \SRDARR, \RDARR, \SBUARR, \BUARR\} \\
%    
%    \tau(\Psi \star \Phi) = \tau'(\Psi \star \Psi) = \{\star\} \cup \tau'(\Psi) \cup \tau'(\Phi) &
%    \text{ for every binary connective } \star \\
%    
%    \tau(p) = \tau'(p) = \{p\} &
%    \tau(n) = \tau'(n) = \{n\}
%\end{array}
%\end{equation}
%
%For example, $\tau((\sbuar (m \mdrarr q) \mor n) \mlarr (p' \mand q')) = \{\mor, n, \mlarr, p', \mand, q'\}$. 
%

The language expansion with $\ell$-$r$-variants guarantees that \fDLG enjoys the display property. Indeed, any substructure, no matter if it occurs in a positive, negative or neutral sequent, can be isolated either in precedent or, exclusively, in succedent position. This property is desirable when it comes to prove cut-elimination or develop a general theory for a class of calculi. Nevertheless, allowing structural rules in positive or negative sequents has undesirable consequences on focalization. We argue that confining to $\ell$-$r$-variants-free proofs is harmless in the following sense:
\begin{proposition}
for every \fDLG-derivable sequent $S$ there exists an equivalent $\ell$-$r$-variants-free sequent $S'$ such that $S'$ has a $\ell$-$r$-variants-free proof. 
\end{proposition}
\begin{proof}
We provide here a sketch of the proof. First of all, notice that display postulates are invertible unary structural rules, no other structural rules is allowed in positive or negative sequents, and auxiliary formulas in tonicity rules occur in isolation. Therefore, even tough we may apply a series of display postulates, what we get are equivalent sequents that can be further manipulated only by applications of display postulates. Therefore, the proof search boils down to retrieve back the initial sequent of this list of equivalent sequents and continue as planned. 
\end{proof}

We can now state the strong focalization property tailored to \fDLG.

\begin{theorem}\label{thm:cfdlg-strongly-focalizaed}
Every cut-free and $\ell$-$r$-variants-free proof in \fDLG is strongly focalized.
\end{theorem}

\begin{proof}
Fix a cut-free and $\ell$-$r$-variants-free \fDLG-proof $\pi$, a formula $A$ occurring a sequent of $\pi$, and a PIA subtree $\Sigma$ of $A$. We prove by induction on $\Sigma$ that for every subtree $\Sigma'$ of $\Sigma$ which is closed by descendent, the subgraph of $\pi$ formed by the rules introducing the connectives of $\Sigma'$ is a proof-section of $\pi$ of end sequent $S$, and if $\Sigma' \neq \Sigma$ then $S$ is of the form $(*):$ $X \rvd P$ or $N \bvd \Delta$.

Call $h$ the root of $\Sigma'$ and $R$ the rule introducing $h$ in $\pi$. We decompose $\Sigma' = h(\Sigma_1,..., \Sigma_n)$, with $n \in \{1,2\}$ ($h$ is a shift or LG connective) and $\Sigma_i$ a subtree closed by descendent. As $h$ is a PIA node, $R$ is a tonicity rule by proposition \ref{prop:connective-introduction-generation-tree}.

Case (a): If $\Sigma_i$ is empty, we let $\pi_i$ be the tree consisting of the $i$-th premise $S_i$ of $R$. As $S_i$ is derivable, $\pi_i$ is a proof-section.

Case (b): If $\Sigma_i$ is non-empty, we apply the induction hypothesis on $\Sigma_i$, yielding a proof-section $\pi_i$ of $\pi$ containing only tonicity rules and of end sequent $S_i$ of the form $(*)$.

Take $\pi'$ the subgraph made of $\pi_1$,..., $\pi_n$ and the conclusion $S$ of $R$. In case (a), $S$ is connected to $\pi_i$ by construction. In case (b), by looking at the rules, the only variant-free rules applicable on focused sequents (i.e. sequents of the form $(*)$) are tonicity rules, introducing an operational connective. Therefore, the only possibility is that the rule after $\pi_i$ is $R$, so $S$ is connected to $\pi_i$. Therefore, $\pi'$ is a proof-section containing only tonicity rules and introducing all connectives of $\Sigma'$.

If $\Sigma \neq \Sigma'$, $h$ is not the root of $\Sigma$. Therefore, $h$ is not a transition node and not the root of $A$, so $h$ is not a shift by proposition \ref{prop:transition-node-shift}. Therefore, $S$ is also of the form $(*)$.

% Assume that there is a proof $\pi$ of end sequent $S$ containing a $\F$-formula $A$ in succedent position wlog. Also assume that there are a subproof $\pi'$ of $\pi''$ and a strict subproof\footnote{The case of $\pi''$ = $\pi'$ is trivial} $\pi''$ of $\pi'$, the last rules of $\pi'$ and $\pi''$ introducing a connective $h'$ and $h''$ of $\tau(A)$ respectively. Call $S''$ the end sequent of $\pi''$, $A''$ the subformula of $A$ of root $h''$ and $R$ any rule between $S''$ and the last proof $R'$ of $\pi'$.

% By assumption, $h''$ lies in the subtree of the signed generation tree of $A$ containing the root. Therefore, it is not a shift node by \ref{prop:transition-node-shift}. Moreover, $h''$ is not the root of $A$ because $\pi''$ is a strict subproof of $\pi'$, hence $h''$ is a LG connective. Therefore, $S''$ being the conclusion of a tonicity rule of a LG connective, we must have $S'' = X \rvd A''$ or $S = A'' \bvd \Delta$. The only rules applicable on $S''$ are tonicity rules, either of a LG connective, going back on one of the same two possibility for $S''$, or a shift, all belonging to $\tau(A)$. Yet no node between $h''$ and the root of $A$ is a transition node, so no one of them is labelled by a shift (except possibly the root) by proposition \ref{prop:transition-node-shift}. Let us call $\S$ the section consisting of the successive tonicity rules below $S''$.

% If the rule $R'$ introducing $h'$ did not belong to $\S$, then there would be a node labelled by a shifts strictly between $h'$ and $h''$: contradiction. So $R'$ belongs to $\S$ and so does $R$.
\end{proof}

\begin{proposition}\label{prop:no-empty-pia}
Every PIA subtree of a formula occurring in a variant-free \fDLG-sequent contains at least one LG-connective.
\end{proposition}

\begin{proof}
This is due to the fullness of the polarization, i.e. the sort of shifts. The target of $\sbuar$ and $\srdar$ is shifted but their source sort is pure, i.e. their argument must begin by a LG formula or an atom. In other words, composing $\sbuar$ and $\srdar$ is impossible.
\end{proof}

Proposition \ref{prop:no-empty-pia} forces the focused sections to be uninterrupted from the point of view of LG connectives. In a forward-looking derivation, it is then impossible to defocus a formula $A$, and then refocus on $A$. Therefore, translating a \fDLG derivation to \fLG by removing shift rules would preserve strong focalization.

Now we provide the definition of phases and phase transitions tailored to \fDLG.

\begin{definition}[Phases and phase transitions]
\label{def:phase-transition}
Let $\pi$ be a cut-free and $\ell$-$r$-variant-free proof in \fDLG. A sequent $S$ occurring in $\pi$ is \emph{focused} (aka $S$ is in a focused phase of $\pi$) if it is positive ($S = \GX \rvD \GY$) or negative ($S = \GD \bvD \GG$) and no structural shift occurs in $S$ (namely, $\SRDARR,\RDARR, \SBUARR,\BUARR$). Any other sequent $S'$ occurring in $\pi$ is \emph{non-focused} (aka $S'$ is in a non-focused phase of $\pi$).

A \emph{phase transition} in $\pi$ is a proof-section $\pi'$ of $\pi$ such that the LG-connectives tonicity rules are not applied in $\pi'$ and its initial-sequent is focused (resp.~non-focused) iff its end-sequent is non-focused (resp.~focused). A phase transition where the initial-sequent is focused is called \emph{defocusing}, and \emph{focusing} otherwise. 
\end{definition}

By design of \fDLG, the application of a shift logical rule is needed to move from a focused to a non-focused phase (resp.~from a non-focused to a focused phase). Therefore, we may say that principal shifted formulas are the gate-keepers of phase transitions. Because \fDLG enjoys subformula property, any formula introduced in a cut-free proof $\pi$, and in particular shifted formulas, will also occur in the conclusion of $\pi$. Therefore, we may say that shifted formulas are \emph{witnesses} of the relevant proof structure of $\pi$. Therefore, we find useful to introduce the following nomenclature:
\begin{definition}[Entry-point and exit-point]
\label{def:EntryExitPoints}
The principal formula introduced by $\srdar_L$ (resp.~$\sbuar_R$) is called the positive (resp.~negative) \emph{entry-point} of the induced phase transition. The principal formula introduced by $\srdar_R$ (resp.~$\sbuar_L$) is called the positive (resp.~negative) \emph{exit-point} of the induced phase transition. 
\end{definition}

We now provide a diagrammatic representation to visualise the topology of rules and phase transitions tailored to \fDLG and, ultimately, make apparent the strong focalization property.

The diagram in figure \ref{fig:seq-graph-schema} depicts the phase transition flow chart of \fDLG. The white area contains the generic form of focused sequents, where each of them is either positive or negative (see definition \ref{def:phase-transition}). The grey and yellow areas contain the generic form of non-focused sequents (see definition \ref{def:phase-transition}), where all neutral sequent seat inside the yellow area.\footnote{Notice that sequents of the form $\SX \urvvd Y$, $\Delta \ubvvd \SG$ and $\SX \vvvd \SD$ are not derivable (see proposition \ref{prop:NonDerivableSequents}) and, therefore, they are not included in the diagram.} We use metavariables $S, S', S''$ for sequents and arrows to depict rules. Arrows $\xrightarrow{R}$ and ${\circ\!\!\xrightarrow{R}}$ points towards the conclusion of the rule $R$. In particular, $\xrightarrow{R} S$ represents a zeroary rule $R$ (i.e.~an axiom) and $S' \xrightarrow{R} S$ represents a unary rule $R$ (here shift logical rules). A double-headed arrow $\xleftrightarrow{R}$ represents an invertible rule (here structural rules introducing or eliminating a shift).  $ \multimap$ are `teleporters' where the configuration ${S'\multimap}$ and ${S''\multimap}$ together with ${\circ\!\!\rightarrow\, }S$ represents a binary rule with premises $S', S''$ and conclusion $S$ (i.e.~tonicity rules).\footnote{Notice that in this case we do not explicitly mention the name of the rule in the diagram.} To exemplify the conventions involving teleporters, let us consider two configurations included in the diagram of figure \ref{fig:seq-graph-schema}. (i) Sequents of the form $X \rvd P$ could occur as premises and conclusion of $\mand_R$, therefore they occurs in the configuration ${X \rvd P\multimap}$ and ${X \rvd P\multimap}$ together with ${\circ\!\!\to}\, X \rvd P$. (ii) Sequents of the form $X \rvd P$ could occur as premise of $\mrarr_L$ and sequents of the form $N \bvd \Delta$ could occur as premise and conclusion of $\mrarr_L$, therefore they occur in the configuration ${X \rvd P\multimap}$ and ${N \bvd \Delta\multimap}$ together with ${\circ\!\!\rightarrow\, N \bvd \Delta}$.

Summing up, the topology of rules is a follows: (i) the white area is closed under axioms, tonicity rules, and display postulates for $\ell$-$r$-variants, (ii) the gray area is closed under display postulates for shifts and $\ell$-$r$-variants, (iii) the yellow area is closed under any other structural rules (i.e.~display postulates for LG-connectives and, whenever we consider analytic extensions of the minimal logic, all the relevant additional structural rules) and translation rules, (iv) the boundary between white and gray areas is crossed only by (non-invertible) shift logical rules, and (v) the boundary between gray and yellow area is crossed only by (invertible) shift structural rules. 

\begin{figure}[ht]
\begin{center}
\hspace{-0.5cm}
\begin{tikzpicture}[x=12.5ex, thick,
inner sep=0.3ex, outer sep=1pt, minimum size=0ex, label distance=-0.5ex]
\draw[fill=gray!20] (0.45,-2) rectangle (3.5,-10);
\draw[fill=yellow!20] (1.6,-4.55) rectangle (2.45,-7.47);
\draw (3,-2.5) node(xsd'){$X \rvvd \SRDARR \Delta$};
\draw (2,-2.5) node(xRDARRd){$X \rvd \RDARR \SD$};

\draw (2,-3.5) node(sxsp){$\SBUARR X \bvvvd\! \sbuar P$};
%
% pp |- pp
\draw (-0.4,-3.5) node[](mxp){$X \rvd P$};
% pp |- ns
\draw (2,-5) node[](xd'){$X \vvdt \SD$};
\draw (3.1,-5) node[](SPD'){$\SBUARR P \bvvvd \SD$};
\draw (3.1,-6) node[](SPD){$\SBUARR P \bvvd \Delta$};
% sn |- sn
\draw (4.3,-5) node[](sPD'){$\sbuar P \bvvvd \SD$};
% pp |- pn
\draw (2,-6) node[](xd){$X \vd \Delta$};
\draw (0.9,-6) node[](xSN){$X \rvvd \SRDARR N$};
\draw (0.9,-7) node[](x'SN){$\SX \rvvvd \SRDARR N$};
% pp |- sp
\draw (-0.4,-6) node[](mxp'){$X \rvvd \srdar N$};
% sn |- pn
\draw (4.3,-6) node[](sPD){$\sbuar P \bvvd \Delta$};
% sp |- pn
\draw (2,-7) node[](x'd){$\SX \tvvd \Delta$};
% sp |- sp
\draw (-0.4,-7) node[](mx'p'){$\SX \rvvvd \srdar N$};
% pn |- pn
\draw (4.3,-8.5) node[](nmd){$N \bvd \Delta$};
\draw (2,-8.5) node(snsd){$\srdar N \rvvvd\! \SRDARR \Delta$};
\draw (2,-9.5) node(sBUARRd){$\BUARR \SX \bvd\! \Delta$};
\draw (1,-9.5) node(snd){$\SBUARR X \bvvd \Delta$};
% edges
%%
\draw[->] (-1.3,-3.5) -- node[above]{$p$-Id} (mxp);
\draw[->] (5.2,-8.5) -- node[above]{$n$-Id} (nmd);
\draw[o->] (-0.4,-2.5) node[above]{} -- (mxp);
\draw[o->] (4.3,-9.5) node[below]{} -- (nmd);
\draw[-o] (sPD') -- (5,-5) node[above]{};
\draw[-o] (mxp) -- (-0.4,-4.5) node[below]{};
\draw[-o] (mxp') -- (-1.1,-6) node[above]{};
\draw[-o] (mx'p') -- (-1.1,-7) node[above]{};
\draw[-o] (sPD) -- (5,-6) node[above]{};
\draw[-o] (nmd) -- (4.3,-7.5) node[above]{};
\draw[->] (nmd) -- node[pos=.23,above]{$\srdar_L$} (snsd);
\draw[->] (mxp) -- node[pos=.25,above]{$\sbuar_R$} (sxsp);
\draw[->] (sxsp) -| node[pos=.33,above]{$\sbuar_L$} (sPD');
\draw[->] (snsd) -| node[pos=.32,above]{$\srdar_R$} (mx'p');
\draw[->] (xSN) -- node[pos=.4,above]{$\srdar_R$} (mxp');
\draw[->] (x'SN) -- node[pos=.4,above]{$\srdar_R$} (mx'p');
\draw[->] (SPD') -- node[pos=.4,above]{$\sbuar_L$} (sPD');
\draw[->] (SPD) -- node[pos=.4,above]{$\sbuar_L$} (sPD);
\draw[<->] (sxsp) edge node[left]{$\SBUARR$} (xd');
\draw[<->] (snsd) edge node[right]{$\SRDARR$} (x'd);
\draw[<->] (xd) -- node[pos=.6,above]{$\SBUARR$} (SPD);
\draw[<->] (x'd) -- node[pos=.5,above]{$\SRDARR$} (x'SN);
\draw[<->] (xd) -- node[pos=.5,above]{$\SRDARR$} (xSN);
\draw[<->] (xd') edge node[pos=.6,above]{$\SBUARR$} (SPD');
\end{tikzpicture}
    \caption{\ \ The topology of \fDLG-rules and phase transitions.}
    \label{fig:seq-graph-schema}
\end{center}
\end{figure}
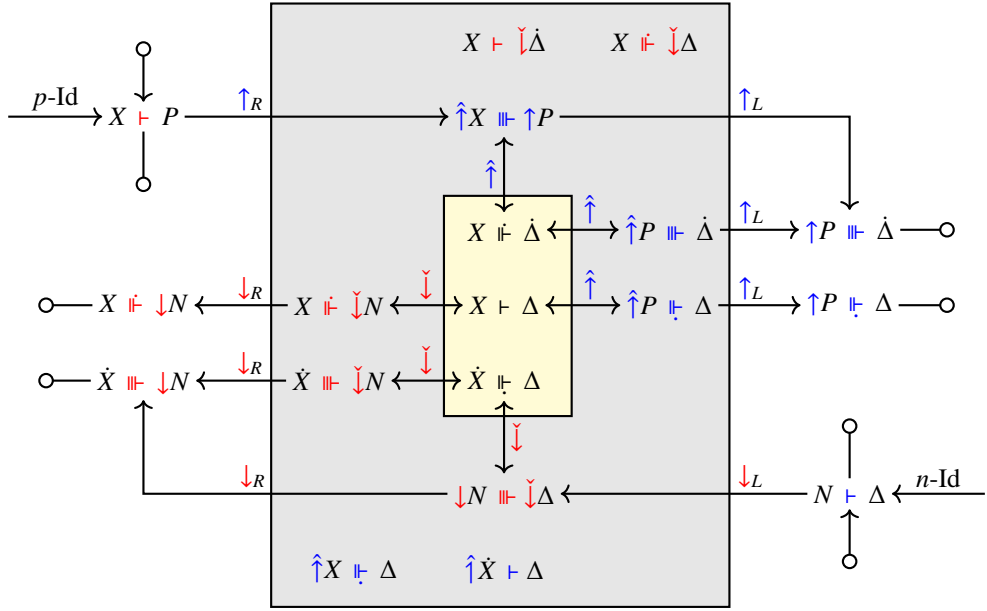

%\begin{proposition}[Witness property]
%\label{prop:de-focusing-fdlg}
%Let $\pi$ be a cut-free and $\ell$-$r$-variant-free proof in \fDLG. A phase transition in $\pi$ is always induced by shift logical rules. Moreover, any operational shift introduced in $\pi$ will occur in the end sequent of $\pi$.  
%\end{proposition}
%
%\begin{proof}
%The proof is by induction on \fDLG-proofs and by %case analysis of applicable rules. 
%\end{proof}

%

%%%%%%%%%%
\section{Completeness of focusing}\label{sec:completeness-of-focusing}

In this section first we prove that the focused calculus \fDLG is sound and complete w.r.t.~$\FPLG$-. Then we prove that \fDLG is sound and complete w.r.t.~LG-algebras: this amounts to a semantic argument showing the so-called completeness of focusing. 

\subsection{Soundness and completeness w.r.t.~\texorpdfstring{$\FPLG$}{FPLG}-algebras}\label{sec:soundness-completeness}

Soundness and completeness are proven as usual in the case of algebraic semantics, where the only departure is that now we consider weakening relations instead of just orders. Soundness is stated as follows:  

%%%%%
%\subsection{Soundness}

\begin{theorem}\label{thm:soundness}
Each rule of the focused display calculus \fDLG is sound under any interpretation in a fully polarized algebras $\FPLG$.
\end{theorem}

\begin{proof}
Given a $\FPLG$-algebra $\bbA$ and an interpretation $(\cdot)^\bbA$, it is straightforward to check by induction on the complexity of proofs that for every sequent $S$ derivable in \fDLG, the interpretation $(S)^\bbA$ is valid. We leave the proof to the reader. Below we simply recall that interpretations of pure atomic formulas $p^\bbA$ and $n^\bbA$ homomorphically extend to arbitrary formulas, and each consequence relation is interpreted by a weakening relation as follows

\begin{equation}\label{eq:turnstile-interpretation}
\begin{array}{|c|c|c|c|c|c|c|c|c|c|c|c|c|} \hline
     t & \nrvd &  \nrvvd & \nrvvvd &
        \nbvd & \nbvvd &  \nbvvvd &
        \nvd & \nvvd & \vvdn & 
        \rvD & \bvD & \vD \\ \hline
    
    t^\bbA & \arvd &  \hrvvd & \arvvvd &
        \abvd & \hbvvd &  \abvvvd &
        \hvd & \hvvd & \vvdh & 
        \arvD & \abvD & \hvD \\ \hline
\end{array}
\end{equation}
\end{proof}

%%%%%
%\subsection{Completeness}
% \begin{definition}
% The principal skeleton of a structure $\Psi$ is a skeleton sub-tree the root of which coincides with the root of $\Psi$. Notice that the principal skeleton could be empty and, if not, every node in it does not occur under any PIA node.
% \end{definition}

In order to prove completeness, we need to introduce the auxiliary notion of standard sequents. 

\begin{definition}
The principal subtree of a structure $\Psi$ is the largest subtree of the signed generation tree of $A$ containing the root and which is either a skeleton subtree or a PIA subtree.
\end{definition}

% \begin{definition}%\label{def:ftom}
%  Let $\Psi$ be a structure. We call $\Ftom{\Psi}$ (resp.~$\FtoM{\Psi}$) the structure of same sort obtained, when it is defined, by turning every connective of its principal skeleton into its structural counterpart.
% %\begin{romanenumerate}
% %    \item its structural (resp.~operational) counterpart if $\Psi$ is a $\F$-structure or formula
% %    \item its operational (resp.~structural) counterpart in $\Psi$ is a $\G$-structure or formula
% %\end{romanenumerate}
% \end{definition}

\begin{definition}\label{def:ftom}
Let $\Psi$ be a structure. We call $\Ftom{\Psi}$ (resp.~$\FtoM{\Psi}$) the structure of same sort obtained, when it is defined, by turning every connective of its principal subtree $\Sigma$ into either
\begin{romanenumerate}
    \item its structural counterpart if $\Sigma$ is a skeleton subtree of $+\Psi$ (resp. $-\Psi$)
    \item its operational counterpart if $\Sigma$ is a PIA subtree of $+\Psi$ (resp. $-\Psi$)
\end{romanenumerate}
and turning all other connectives into their operational counterpart.

Given a well-formed sequent $\Psi \t \Phi$, its \textbf{standard sequent} is $\Ftom{\Psi} \t \FtoM{\Phi}$.
\end{definition}

To exemplify the instrumental use of standard sequents in proving completeness, consider the following observation. The sequent $p \mand q \rvd p \mand q$ is not derivable in \fDLG, despite the fact that $\arvd$ is a partial order, so in particular $p^{\P} \mand^{\P} q^{\P} \arvd p^{\P} \mand^{\P} q^{\P}$ in every $\FPLG$-algebra. However, the standard sequent $p \MAND q \rvd p \mand q$ is \fDLG-derivable and moreover $(p \MAND q)^{\P} = (p \mand q)^{\P} = p^{\P} \mand^{\P} q^{\P}$. See lemma \ref{lem:prop-ftom} in appendix \ref{ap:proofs} for a recursive definition of $\Ftom{~}$ and $\FtoM{~}$. Completeness is stated as follows:

\begin{theorem}\label{thm:completeness}
For every $\FPLG$-algebra $\bbA$ and every well-formed sequent $\Psi \t \Phi$ in the language of \fDLG, if the interpretation $(\Psi \t \Phi)^\bbA$ is a valid, then the standard sequent $\Ftom{\Psi} \t \FtoM{\Phi}$ is derivable in \fDLG.
\end{theorem}

\begin{proof}
We prove completeness by building a syntactic model $\bbA$. Let $\approx_{\P}$, $\approx_{\sP}$, $\approx_{\N}$ and $\approx_{\sN}$ be the equivalence relation generated by \eqref{eq:approx}.

\begin{equation}\label{eq:approx}
\begin{array}{ccccc}
    \Psi \approx_{\P} \Phi & \text{ iff } & \Ftom{\Psi} \rvd \FtoM{\Phi} & \text{ and } &
    \Ftom{\Phi} \rvd \FtoM{\Psi} \\
    
    \Psi \approx_{\sP} \Phi & \text{ iff } & \Ftom{\Psi} \rvvvd \FtoM{\Phi} & \text{ and } &
    \Ftom{\Phi} \rvvvd \FtoM{\Psi} \\
    
    \Psi \approx_{\N} \Phi & \text{ iff } & \Ftom{\Psi} \bvd \FtoM{\Phi} & \text{ and } &
    \Ftom{\Phi} \bvd \FtoM{\Psi} \\
    
    \Psi \approx_{\sN} \Phi & \text{ iff } & \Ftom{\Psi} \bvvvd \FtoM{\Phi} & \text{ and } &
    \Ftom{\Phi} \bvvvd \FtoM{\Psi}
\end{array}
\end{equation}

\noindent In particular, $\approx_{\P}$, $\approx_{\sP}$, $\approx_{\N}$ and $\approx_{\sN}$ are congruence relations (by tonicity rules and cut rules, see appendix \ref{ap:proofs} for a detailed proof). For any $s \in \{\P, \sP, \N, \sN\}$, let $[\Psi]_{\approx_{s}}$ denote the class of structures $\Phi$ such that $\Phi \approx_s \Psi$. We define operations and weakening relations by \eqref{eq:op-and-wr-on-eq-classes}.

\begin{equation}\label{eq:op-and-wr-on-eq-classes}
\begin{array}{c}
\begin{array}{ccc}
     [\GX]_{\approx_{\gP}} \mand^{\bbA} [\GY]_{\approx_{\gP}} = [\GX \MAND \GY]_{\approx_{\P}} & 
     [\GX]_{\approx_{\gP}} \mdlarr^{\bbA} [\GD]_{\approx_{\gN}} = [\GX \MDLARR \GD]_{\approx_{\P}} &
     [\GD]_{\approx_{\gN}} \mdrarr^{\bbA} [\GD]_{\approx_{\gP}} = [\GD \MDRARR \GY]_{\approx_{\P}} \\
     
     [\GD]_{\approx_{\gN}} \mor^{\bbA} [\GG]_{\approx_{\gN}} = [\GD \MOR \GG]_{\approx_{\N}} & 
     [\GX]_{\approx_{\gP}} \mrarr^{\bbA} [\GD]_{\approx_{\gN}} = [\GX \MRARR \GD]_{\approx_{\N}} &
     [\GD]_{\approx_{\gN}} \mlarr^{\bbA} [\GD]_{\approx_{\gP}} = [\GD \MLARR \GY]_{\approx_{\N}}
\end{array} \\[3mm]
\begin{array}{cccc}
     \srdar^{\bbA} [\Delta]_{\approx_{\N}} = [\SRDARR \Delta]_{\approx_{\sP}} &
     \rdar^{\bbA} [\SD]_{\approx_{\sN}} = [\RDARR \SD]_{\approx_{\P}} &
     \sbuar^{\bbA} [X]_{\approx_{\P}} = [\SBUARR X]_{\approx_{\sN}} &
     \buar^{\bbA} [X]_{\approx_{\gP}} = [\BUARR X]_{\approx_{\N}} 
\end{array} \\
    \text{and similar for the $\ell,r$-variants} \\[2mm]
     \text{or a turnstile $t$ from sort $s$ to $s'$, }
     [\Psi]_{\approx_s} \ t^{\bbA}\ [\Phi]_{\approx_{s'}} \text{iff }
     \Ftom{\Psi} \t \FtoM{\Phi} \text{ is derivable}
\end{array}
\end{equation}

It is not difficult to see that the operations and relations of \eqref{eq:op-and-wr-on-eq-classes} are well-defined, and that the relations are indeed orders or weakening relations. The technical proof is provided in appendix \ref{ap:proofs}.

We take $\P = \mathsf{PurePosStr} / \!\approx_{\P}$, $\sP = \mathsf{ShiftedPosStr} / \!\approx_{\sP}$, $\N = \mathsf{PureNegStr} / \!\approx_{\N}$ and $\sN = \mathsf{ShiftedNegStr} / \!\approx_{\sN}$. It is easy to show that properties \eqref{eq:shifts-fplg}, \eqref{eq:shift-intro-elim}, \eqref{eq:adj-fplg} and \eqref{eq:adj-fplg-variants} of definition \ref{def:fplg} are verified thanks to the correspondent rules of the calculus.
\end{proof}

\subsection{Soundness and completeness w.r.t.\ LG-algebras}

Given an LG-algebra $\bbG = (G, \leq, \mand^\bbG, \mlarr^\bbG, \mrarr^\bbG, \mor^\bbG, \mdlarr^\bbG, \mdrarr^\bbG)$ we define an $\FPLG$ algebra $\bbA_\bbG$ as follows: We take a copy of $G$ as the domain of any sub-algebra in $\bbA_\bbG$, by defining shifts as maps sending an element to its copy in the appropriate sub-algebra of $\bbA_\bbG$, and finally, for each $A, B$ in the appropriate sub-algebra of $\bbA_\bbG$, for each weakening relation $R$ in $\bbA_\bbG$, and for each binary operation $\star^{\bbA_\bbG}$ in $\bbA_\bbG$, by defining 
\begin{center}
	$A \,R\, B$ \ iff \ $A \leq B$ \ \ \ and \ \ \ $A \star^{\bbA_\bbG} B$ \ iff \ $A \star^\bbG B$
\end{center}

\begin{proposition}
For every LG-algebra $\bbG$, $\bbA_\bbG$ is an $\FPLG$-algebra. 
\end{proposition}
\begin{proof}
It is straightforward to check that weakening relations and operations are well-defined, and \eqref{eq:shift-intro-elim}, \eqref{eq:adj-fplg}, and \eqref{eq:adj-fplg-variants} hold, so $\bbA_\bbG$ is an $\FPLG$ algebra accordingly to Definition \ref{def:fplg}.
\end{proof}

Conversely, given an $\FPLG$ algebra $\bbA$ we first define $\pi(\bbA)= (L, \leq, \mand^\bbG, \mlarr^\bbG, \mrarr^\bbG, \mor^\bbG, \mdlarr^\bbG, \mdrarr^\bbG)$ by taking $L = \P \sqcup \N $, by defining $A \leq B$ iff $A^+ \hvD B^-$, where 
\begin{equation}
	\gP \ni \ A^+ = \left\{\begin{array}{ll}
		A & \text{if } A \in \P \\
		\srdar A & \text{if } A \in \N 
	\end{array} \right.
	\text{\ \ \  and \ \ \ \ }
	\gN \ni \ A^- = \left\{\begin{array}{ll}
		A & \text{if } A \in \N \\
		\sbuar A & \text{if } A \in \P
		\end{array} \right.
\end{equation}
 and by defining the operations as follows
 \begin{equation}
 	\begin{array}{ccc}
 		A \mand^{\pi(\bbA)} B := A^+ \mand^{\bbA} B^+ &
 		A \mdlarr^{\pi(\bbA)} B := A^+ \mdlarr^{\bbA} B^- &
 		A \mdrarr^{\pi(\bbA)} B := A^- \mdrarr^{\bbA} B^+ \\
 		
 		A \mor^{\pi(\bbA)} B := A^- \mor^{\bbA} B^- &
 		A \mrarr^{\pi(\bbA)} B := A^+ \mrarr^{\bbA} B^- &
 		A \mlarr^{\pi(\bbA)} B := A^- \mlarr^{\bbA} B^+. 
 	\end{array}
 \end{equation}

\begin{proposition}
For every $\FPLG$ algebra $\bbA$, $\pi(\bbA)$ is a pre-order.
\end{proposition}
\begin{proof}
First let us show that $\leq$ is transitive. Assume that $A \leq B$ and $B\leq C$, that is $A^+ \hvD B^-$ and $B^+ \hvD C^-$. If $B\in \P$ then $B\hvD C^-$, which is equivalent to $\sbuar B \abvD C^-$ and hence $A^+ \hvD C^-$, i.e. $A\leq C$. If $B\in \N$ then $A^+ \hvD B$, which is equivalent to $ A^+\arvD \srdar B$ and hence again we get $A \leq C$. It is easy to show that $\leq$ is reflexive, so $\leq$ is a pre-order.
\end{proof}

Now we define $\bbG_\bbA$ based on $\pi(\bbA)$ by taking the quotient over $\leq\cap\geq$. Since the operations on $\pi(\bbA)$ are monotone and antitone, $\leq\cap\geq$ is in fact a congruence relation and the operations on $\bbG_\bbA$ can be defined in the usual way. 

\begin{proposition}
For every $\FPLG$ algebra $\bbA$, $\bbG_\bbA$ is an LG-algebra.
\end{proposition}
\begin{proof}
We need to show that the defined operations are residuated in each coordinate. Assume that $A\in \P, B\in\N$ and $C\in \P$:

\begin{equation}\notag
	\begin{array}{rcl}
		A \mand^{\bbG_\bbA} B \leq C & \text{ iff } & 
		A \mand^{\bbA} \srdar B \,\hvD\, \sbuar C \\
		& \text{ iff } & A \,\hvD\, \sbuar C \mlarr^{\bbA}\, \srdar B \\
		& \text{ iff } & A \leq C \mlarr^{\bbG_\bbA}\, B
	\end{array}
\end{equation}
and 
\begin{equation}\notag
	\begin{array}{rcl}
		A \mand^{\bbG_\bbA} B \leq C & \text{ iff } & 
		A \mand^{\bbA} \srdar B \,\hvD\, \sbuar C \\
		& \text{ iff } & \srdar B \,\hvD\, A \mrarr^{\!\bbA}\,\sbuar C\\
		& \text{ iff } & B \leq A \mrarr^{\!\bbG_\bbA}\, C
	\end{array}
\end{equation}
The rest of the cases are done analogously.
\end{proof}	

\begin{theorem}[Completeness and Soundness]
The system \fDLG is sound and complete with respect to LG-algebras.
\end{theorem}

First let us define a translation of \emph{formulas} of \fDLG into formulas of the language of LG-algebras. We do so recursively:

\begin{itemize}
	\item Positive and negative atoms are sent to atoms.
	\item For each binary connective $\star$ we define $\tau(A\star B)$ to be $\tau(A) \star \tau(B)$.
	\item Finally $\tau(\srdar A)$ and $\tau(\sbuar A)$ are defined to be $\tau(A)$.
\end{itemize}

Let $t$ be an arbitrary turnstile in the language of \fDLG and let $w$ be an arbitrary weakening relation of an $\FPLG$ algebra. We will show that any sequent $A \ t\  B$ is provable in \fDLG if and only if $\tau(A) \vdash\tau(B)$ is provable in the logic of LG-algebras. Since \fDLG is sound and complete with respect to $\FPLG$ algebras it is enough to show that the sequent is falsified in an  $\FPLG$ algebra if and only if it's translation is falsified in an LG-algebra. 

Assume $\bbG\nvDash \tau(A)\leq \tau(B)$. Then it is immediate that $\bbA_\bbG \nvDash A \ w\  B$ (since $\srdar$ and $\sbuar$ are essentially `identity maps', given that we defined shifts as maps sending an element to its copy in the appropriate sub-algebra of $\bbA_\bbG$).

For the opposite direction first we make a distinction. We call a formula in normal form, if the outermost connective is binary. It's immediate that it's enough to restrict ourselves to normal form formulas and show that if $\bbA\nvDash A \ w\  B$ then $\bbG_\bbA\nvDash \tau(A)\leq \tau(B)$. Notice that if in $\pi(\bbA)$, it is the case that $C\nleq D$ then so is the case in $\bbG_\bbA$. So assume  that  $\bbA\nvDash A \ w\  B$. Then by definition $\pi(\bbA)\nvDash \tau(A) \leq \tau(B)$. This in turn implies that $\bbG_\bbA\nvDash \tau(A) \leq \tau(B)$. This concludes the proof.

%%%%%%%%%% Cut elimination %%%%%
\section{Canonical cut-elimination}
\label{ap:heterogeneous-dc}

We show here that the class of multi-type proper display calculi \cite{FriGreKurPalSik16} can be extended to include calculi involving heterogeneous sequents, and that \fDLG belongs to that class.

\begin{remark}
For simplicity, here we confine ourselves to `minimal' multi-type proper display calculi with heterogeneous sequents, i.e.~calculi where the only structural rules are display postulates and cuts. If we admit additional structural rules in the neutral phase, then some additional care is needed. We leave this as future work. 
\end{remark}

Eliminating a parametric (possibly heterogeneous) cut amounts to be able to substitute a formula of a sort $s$ by any structure of another (possibly different) sort $s'$, and keeping derivability. As structural connective arguments have a fixed sort, substitution may lead to a clash of sorts.

To illustrate how this has to work, let us take the example of \eqref{eq:ex-cut-elim}, where we want to move up the cut on the derivable sequent $p \MAND (\srdar p \mrarr n) \rvvd \srdar n$ to the uppermost occurrence of $\srdar n$ in $\pi_2$ (see \eqref{eq:ex-pi2}). This transformation requires to substitute every parametric occurrence of $\srdar n$ in $\pi_2$ by $p \MAND (\srdar p \mrarr n)$, which is still positive but pure. The problem is that there is an occurrence of $\srdar n$ under $\BUARR$, and that this connective only takes shifted structures as argument. Therefore, we have to also mutate (i.e. convert) $\BUARR$ into $\SBUARR$ so that the sequent stays well-formed. We can check that the instances of rule $\BUARR \dashv \SRDARR$ are changed into instances, that turn out to be instances of rule $\SBUARR \dashv \SRDARR$. In other words, the mutation $\BUARR \mapsto \SBUARR$ preserves the derivability. The result of the parametric move is shown in \eqref{eq:ex-cut-elim-result}

The mutation generated by \eqref{eq:ex-cut-elim} also has an impact on $\MAND$, because $\srdar n$ appears as an argument of $\MAND$ in $\pi_2$. However, as this connective accepts shifted as well as pure arguments, it does not have to mutate, or equivalently, it mutates into itself $\MAND \mapsto \MAND$. Last but not least, turnstiles also have to mutate. From \eqref{eq:ex-cut-elim} to \eqref{eq:ex-cut-elim-result}, we have the following convertions: $\nvvd \mapsto \nvd$, $\nrvvvd \mapsto \nrvvd$ and $\nbvd \mapsto \nbvvd$. This example is the pattern $(\Pos, \Shifted) \xrightarrow{pre} (\Pos, \Pure)$ contained in $\mu_{\nrvvd, \nbvvd}$ \eqref{eq:mu-rvvd-bvvd} of proposition \ref{prop:mutations-fdlg}.

\begin{equation}\label{eq:ex-cut-elim}
    \AXC{$\pi$}
    \noLine
    \UIC{$p \MAND (\srdar p \mrarr n) \rvvd \srdar n$}
        \AXC{$\pi_2$}
        \noLine
        \UIC{$\srdar n \tvvd \sbuar (\srdar n \mand p) \MLARR p$}
    \RL{P-Cut}
    \BIC{$p \MAND (\srdar p \mrarr n) \vd \sbuar (\srdar n \mand p) \MLARR p$}
    \DP
\end{equation}

\begin{equation}\label{eq:ex-pi2}
    \pi_2 = \qquad
    \AXC{$\pi_{2.1}$}
    \noLine
    \UIC{$\srdar n \rvvvd \SRDARR n$}
    \noLine
    \UIC{$\vdots$}
    \noLine
    \UIC{$\srdar n \MAND p \vvdt \sbuar (\srdar n \mand p)$}
    \RL{$\MAND \dashv \MLARR$}
    \UIC{$\srdar n \tvvd \sbuar (\srdar n \mand p) \MLARR p$}
    \RL{$\SRDARR$}
    \UIC{$\srdar n \rvvvd \SRDARR (\sbuar (\srdar n \mand p) \MLARR p)$}
    \RL{$\BUARR \dashv \SRDARR$}
    \UIC{$\BUARR \srdar n \bvd \sbuar (\srdar n \mand p) \MLARR p$}
    \RL{$\BUARR \dashv \SRDARR$}
    \UIC{$\srdar n \rvvvd \SRDARR (\sbuar (\srdar n \mand p) \MLARR p)$}
    \RL{$\SRDARR$}
    \UIC{$\srdar n \tvvd \sbuar (\srdar n \mand p) \MLARR p$}
    \DP
\end{equation}

\begin{equation}\label{eq:ex-cut-elim-result}
    \rightsquigarrow \qquad
    \AXC{$\pi$}
    \noLine
    \UIC{$p \MAND (\srdar p \mrarr n) \rvvd \srdar n$}
        \AXC{$\pi_{2.1}$}
        \noLine
        \UIC{$\srdar n \rvvvd \SRDARR n$}
    \RL{P-Cut}
    \BIC{$p \MAND (\srdar p \mrarr n) \rvvd \SRDARR n$}
    \noLine
    \UIC{$\vdots$}
    \noLine
    \UIC{$(p \MAND (\srdar p \mrarr n)) \MAND p \vvdt \sbuar (\srdar n \mand p)$}
    \RL{$\MAND \dashv \MLARR$}
    \UIC{$p \MAND (\srdar p \mrarr n) \vd \sbuar (\srdar n \mand p) \MLARR p$}
    \RL{$\SRDARR$}
    \UIC{$p \MAND (\srdar p \mrarr n) \rvvd \SRDARR (\sbuar (\srdar n \mand p) \MLARR p)$}
    \RL{$\SBUARR \dashv \SRDARR$}
    \UIC{$\SBUARR (p \MAND (\srdar p \mrarr n)) \bvvd \sbuar (\srdar n \mand p) \MLARR p$}
    \RL{$\SBUARR \dashv \SRDARR$}
    \UIC{$p \MAND (\srdar p \mrarr n) \rvvd \SRDARR (\sbuar (\srdar n \mand p) \MLARR p)$}
    \RL{$\SRDARR$}
    \UIC{$p \MAND (\srdar p \mrarr n) \vd \sbuar (\srdar n \mand p) \MLARR p$}
    \DP
\end{equation}

We call $\S_{\F}$ (resp. $\S_{\G}$) the set of structural $\F$-connectives (resp. $\G$-connectives), $\S = \S_{\F} \cup \S_{\G}$ (resp. $\S_n$ for connectives of arity $n \geq 0$) and $\T$ the set of turnstiles. The sort-position function $\sortpst$ maps every structural connective and turnstile to its nonempty vector on $\Sort \times \Pst$, where $\Sort = \{\Pos, \Neg\} \times \{\Pure, \Shifted\}$ is the set of sorts and $\Pst = \{pre, suc\}$ the set of positions\footnote{Recall that the $i$-th position of a structural connective $H$ ($0 \leq i \leq \ar(H)$) is given by: $\pst(H,i) = pre$ iff: (i) $i = 0$ and $H \in S_{\F}$, or (ii) $h \in S_{\F}$ and $\epsilon(h,i) = 1$, or (iii) $h \in S_{\G}$ and $\epsilon(h,i) = \partial$.}. The initial pair of sort and position stands for the target of the structural connective, e.g. $\sortpst(\SBUARR) = \langle (\Neg,\Shifted), (\Pos,\Pure) \rangle$. For a turnstile $t$ from sort $s$ to $s'$, we set $\sortpst(t) = \langle (s, pre), (s', suc) \rangle$.

\begin{definition}
A \textbf{mutation} $\mu$ is a function $\mu: \Sort \times \Pst \to \Sort \times \Pst$ together with two other functions (called identically) $\mu : \cup_n (\S_n \times \wp (\llbracket 1, n \rrbracket) \to S \times \wp(\{0\}))$ and $\mu: \T \times \wp (\llbracket 1, 2 \rrbracket) \to \T$ such that:

\begin{enumerate}
    \item for all $(s,d) \in \Sort \times \Pst$, if $\mu(s,d) = (s',d')$ then $d = d'$
    
    \item for all $(H, I) \in \S$, and $\mu(H,I) = (H',I')$ $\sortpst(H') = \mu_{I \cup I'}(\sortpst(H))$
    
    \item for all $(t, I) \in \T$, $\sortpst(\mu(t,I)) = \mu_I(\sortpst(t))$
\end{enumerate}

where $\mu_I(w_0...w_n) = w_0'...w_n'$ with
\begin{math}
w_i' = \ite{\mu(w_i)}{if $i \in I$}{w_i}{if $i \not\in I$}
\end{math}

Moreover, we say that $\mu$ contains a pattern $s \xrightarrow{d} s'$ if $\mu(s,d) = (s',d)$.
\end{definition}

The input set $I \subseteq \llbracket 1, n \rrbracket$ stands for the arguments that have to be mutated. In the above example, we ask for $\mu(\MAND,\{1\})$ because only the first argument of $\MAND$ contains a parametric occurrence of $\srdar n$. Whether the target (index $0$) changes its sort is up to $\mu$. That's why $\mu(H,I)$ has to produce either $\emptyset$ (not changing the target sort of $H$) or $\{0\}$ (possibly changing the target sort of $H$) as additional output. In the example, we have $\mu(\BUARR, \{1\}) = (\SBUARR, \{0\})$, but $\mu(\MAND, \{1\}) = (\MAND, \emptyset)$ stops there the propagation of the need to mutate structural connectives (i.e. the propagation of the $*$ label in definition \ref{def:mutation-sequent}), and thus turnstile $\vvdn$ does not have to be mutated.

The combinatorial behaviour of mutations can be better understood if we see the set of sorts (in Fig.~\ref{fig:fplg-wr} for \fDLG) as a (thin small) category where weakening relations are morphisms and orders the identities. Given a category $\C$, a morphism $f : A \to B$ acts on morphisms $g : C \to A$ by post-composition $f \circ g : C \to B$ (succedent mutation) and on morphisms $g : B \to D$ by pre-composition $g \circ f : A \to D$ (precedent mutation). The action of identities is the identity.

Mutations can extend to sequents and rules, given a set of congruent structure occurrences.

\begin{definition}\label{def:mutation-sequent}
Fix a mutation $\mu$, a sequent $S$ and a set of formula occurrences $(A_j)_j$ of $S$. For any structures $(\Psi_j)_j$ such that for all $j$, $\langle \sort(\Psi_j), \pst(A_j) \rangle = \mu( \sortpst(A_j) )$, we define the uniform substitution of $(A_j)_j$ by $(\Psi_j)_j$  with mutation $\mu$ on $S$ as follows:
\begin{enumerate}
    \item Take the generation tree of $S$ (its root is the turnstile) and proceed inductively, starting from the deepest nodes.
    \begin{enumerate}
        \item  Substitute every $A_j$ by $\Psi_j$ and label it by $*$ if $\mu(\sortpst(A_j)) \neq \sortpst(A_j)$.
        \item For every internal structural node $H$, by calling $I$ its children who are labelled by $*$, replace $H$ by $H'$ from $\mu(H,I) = (H',I')$, and label it by $*$ if $0 \in I'$ and $\mu(\sortpst(H,0)) \neq \sortpst(H,0)$
        \item Similarly for the turnstile, turn $t$ into $\mu(t,I)$.
    \end{enumerate}
\end{enumerate}
The new sequent $\mu_{[(\Psi_j)_j / (A_j)_j]}(S)$ is well-formed.

On any rule $R$ we define its mutation $\mu_{[(\Psi_j)_j / (A_j)_j]}(S)$ by applying the uniform substitution with mutation to every premise and the conclusions.
\end{definition}

In practice, if $I = \emptyset$ for some connective $H$, $\mu((H,I))$ is supposed to be $(H, \emptyset)$. 

\begin{definition}
\label{def:HeterogeneousDC}
We adapt the conditions of \cite{FriGreKurPalSik16} to define the class of \emph{heterogeneous} multi-type proper display calculi by removing $C_9$ and modifying $C_6'$, $C_7'$ and $C_{10}$ as follows:

\begin{enumerate}
    \item[$C_6''$] (Closure under precedent mutations) For every derivable turnstile\footnote{By derivable turnstile, we mean that there is a derivable sequent on that turnstile.} $t \in \T$ of sort $\sort(t) = (s,s')$, there exists a mutation $\mu$ containing a pattern $s' \xrightarrow{pre} s$ such that the mutation of every rule wrt. $\mu$ is derivable in the calculus.
    
    \item[$C_7''$] (Closure under succedent mutations) For every derivable turnstile $t \in \T$ of sort $\sort(t) = (s,s')$, there exists a mutation $\mu$ containing a pattern $s \xrightarrow{suc} s'$ such that the mutation of every rule wrt. $\mu$ is derivable in the calculus.
    
    \item[$C_{10}'$] (Closure under turnstile composition) For every turnstiles $t, t' \in \T$ such that $\sort(t) = (s, s')$ and $\sort(t') = (s', s'')$, there exists a turnstile $t'' \in \T$ such that the following cut is definable and belongs to the system:
    
    \begin{center}
        \AXC{$\Psi \t A$}
        \AXC{$A \t['] \Phi$}
        \RL{$tt'$-Cut}
        \BIC{$\Psi \t[''] \Phi$}
        \DP
    \end{center}
    
    \item[$C_{10}''$] (Uniqueness of turnstiles) There is at most one turnstile by pair of sorts.
\end{enumerate}
\end{definition}

% We can extend that to set of morphisms if they do not have any source or target pairwise in common. A mutation additionally required that structural connectives can also be transported by these actions, and conditions $C_6''$ and $C_7''$ that it preserves derivability.

\begin{theorem}[\textit{Canonical cut-elimination}]\label{thm:canonical-cut-elim}
Any multi-type heterogeneous sequent calculus enjoys cut-elimination.
\end{theorem}

\begin{proof}
We follow the proof of \cite{FriGreKurPalSik16} and we only expand on the parts of which the proof departs from it, assuming principal formulas are in display. In the parametric move, we are in the following situation:

\def\fCenter{}
\begin{center}
    \AX$ \vdots\fCenter~ \pi$
    \noLine
    \UI$\Psi \ t_1\fCenter\  A$
        \AX$ \vdots\fCenter~ \pi_{2.1}$
        \noLine
        \UI$(A_1 \ t_{2.1}\fCenter\ \Phi_1)$
            \AX$\hdots\fCenter$
                \AX$ \vdots\fCenter~ \pi_{2.n}$
                \noLine
                \UI$(A_n \ t_{2.n}\fCenter\ \Phi_n)$
        \noLine
        \TrinaryInf$\ddots \vdots\fCenter \iddots~ \pi_2$
        \noLine
        \UI$A \ t_2\fCenter\ \Phi$
    \RL{$t_1t_2$-Cut}
    \BI$\Psi \ t_3\fCenter\ \Phi$
    \DP
\end{center}

with the $A_i$s being the uppermost congruent occurrences of $A$. We treat the case of $A_i$ and when it is principal (case (1)). Call $s'$ the sort of $\Psi$ and $s$ the sort of $A$. By $C_2'$, $A_i$ is also of sort $s$.

By $C_{10}'$, the calculus contains the $t_1t_{2.1}$-cut. By $C_6''$, there exists a mutation $\mu$ containing a pattern $s \xrightarrow{pre} s'$. For every rule $R$ in the section $\pi_2$, we track the congruent occurrences $(A_j)_j$ of $A$ in $R$ and substitute them uniformly by $\Psi$ with mutation, yielding $\mu_{[(\Psi)_j / (A_j)_j]}(R)$, which is derivable by $C_6''$. By a straightforward induction on $\pi_2$, its mutation $\mu_{[(\Psi)_j / (A_j)_j]}(\pi_2)$ is a well-formed proof of end sequent $\mu_{[(\Psi)_j / (A_j)_j]}(A \t[_2] \Phi) = \Psi \t[_2'] \Phi$, and $t_2' = t_3$ by $C_{10}''$. Moreover, we can cut on $A_1 \t[_{2.i}] \Psi_1$ with $\Psi \t[_1] A$ thanks to $C_{10}'$, and $t_1 \circ t_{2.i} = \mu(t_{2.i})$ by $C_{10}''$.

\def\fCenter{}
\begin{center}
    \AX$ \vdots\fCenter~ \pi$
    \noLine
    \UI$\Psi \ t_1\fCenter\ A$
        \AX$ \vdots\fCenter~ \pi_{2.i}$
        \noLine
        \UI$A_i \ t_{2.i}\fCenter\ \Phi_i$
        \noLine
        \UI$\vdots\fCenter~ \pi_2$
        \noLine
        \UI$A \ t_2\fCenter\ \Phi$
    \RL{$t_1 t_2$-Cut}
    \BI$\Psi \ t_3\fCenter\ \Phi$
    \DP
    \qquad $\rightsquigarrow$ \qquad
    \AX$ \vdots\fCenter~ \pi$
    \noLine
    \UI$\Psi \ t_1\fCenter\ A$
        \AX$ \vdots\fCenter~ \pi_{2.i}$
        \noLine
        \UI$A_i \ t_{2.i}\fCenter\ \Phi_i$
    \RL{$t_1 t_{2.i}$ Cut}
    \BI$\Psi \ t_1 \circ \fCenter t_{2.1}\  \Phi_i$
    \noLine
    \UI$\vdots\fCenter~ \mu_{[(\Psi)_j/ (A_j)_j]}(\pi_2)$
    \noLine
    \UI$\Psi \ t_3\fCenter\ \Phi$
    \DP
\end{center}

Case (2) of \cite{FriGreKurPalSik16} works similarly.
\end{proof}

%%%
\begin{proposition}\label{prop:mutations-fdlg}
\fDLG satisfies conditions $C_6''$ and $C_7''$.
\end{proposition}

\begin{proof}
The 4 mutations\footnote{They could be reorganised differently. We choose to regroup them as much as possible.} of \fDLG are the following (using the symbols of $\FPLG$ instead of explicit sorts, the $I$ and $I'$ are left implicit):

\begin{itemize}
    \item $\mu_{\nrvd, \nrvvvd, \nbvd, \nbvvvd}$ is the identity
    
    \item $\mu_{\nrvvd, \nbvvd}$ :
    
\begin{equation}\label{eq:mu-rvvd-bvvd}
\begin{array}{cccccc}
     \sP \xrightarrow{pre} \P & \P \xrightarrow{suc} \sP &
     \N \xrightarrow{pre} \sN & \sN \xrightarrow{suc} \N \\
     
     \nrvvvd \mapsto \nrvvd & \nrvd \mapsto \nrvvd &
     \nbvd \mapsto \nbvvd & \nbvvvd \mapsto \nbvvd &
     \nvvd \mapsto \nvd & \vvdn \mapsto \nvd \\
     
     \BUARR \mapsto \SBUARR & \BUARR \dashv \SRDARR \mapsto \SBUARR \dashv \SRDARR &
     \RDARR \mapsto \SRDARR & \SBUARR \dashv \RDARR \mapsto \SBUARR \dashv \SRDARR \\
\end{array}
\end{equation}

    $\mu_{\nrvvd, \nbvvd}$ acts on the rest like the identity or doesn't change the connective's name (thanks to name overloading, e.g. $\MAND$ taking either pure or shifted structures as input).
    
    \item $\mu_{\nvvd, \vvdn}$ :
    
\begin{equation}
\begin{array}{cccccc}
     \P \xrightarrow{suc} \sN & \sP \xrightarrow{suc} \N & &
     \N \xrightarrow{pre} \sP & \sN \xrightarrow{pre} \P \\
     
     \nrvd \mapsto \vvdn & \nrvvd \mapsto \nvd & \nrvvvd \mapsto \nvvd &
     \nbvd \mapsto \nvvd & \nbvvd \mapsto \nvd & \nbvvvd \mapsto \vvdn \\
     
     \text{P-Cut} \mapsto \text{PB-Cut} &  &
     \text{P-Cut} \mapsto \text{PB-Cut} &  \\
\end{array}
\end{equation}

    $H_{\ell} \mapsto H$ and $H_r \mapsto H$ for every LG connective $H$. $\mu_{\nvvd, \vvdn}$ acts on the rest like the identity or doesn't change the connective's name
    
    \item $\mu_{\nvd}$ :
    
\begin{equation}
\begin{array}{cc}
     \P \xrightarrow{suc} \N & \N \xrightarrow{pre} \P  \\
     
     \nrvd \mapsto \nvd & \nbvd \mapsto \nvd \\
     
     \text{P-Cut} \mapsto \text{PB-Cut} &
     \text{P-Cut} \mapsto \text{PB-Cut}\\
\end{array}
\end{equation}

     $\mu_{\nvd}$ acts on the rest like identity or doesn't change the connective's name
\end{itemize}
\end{proof}

\begin{theorem}\label{thm:fdlg-is-d-caluclus}
\fDLG is a heterogeneous multi-type proper display calculus.
\end{theorem}

\begin{proof}
\fDLG clearly enjoys condition $C_2$ -- $C_5$, and $C_8$ because it is fully residuated. New conditions $C_{10}'$ and $C_{10}''$ are also easily checked. Finally, $C_6''$ and $C_7''$ hold thanks to proposition \ref{prop:mutations-fdlg}.
\end{proof}
%%%%

\begin{corollary}
\fDLG enjoys canonical cut-elimination.
\end{corollary}

\section{Proof translations}
\label{sec:ProofTranslations}

The rules of the minimal calculus \fLG are provided in Section 3.1 of \cite{Moortgat--Moot:2011}. Here are the translations between \fDLG and \fLG. In the following, we assume that in \fLG, axioms are only possible on atomic formulas. The results remain valid when axioms on every formula are allowed, but $\FtoT[]{\cdot}$ would involve a structural normalization using $\Ftom{\cdot}$ and $\FtoM{\cdot}$ transformations.

%%%%%
\subsection{From \fLG to \fDLG}

\begin{definition}\label{def:polar-form}
Given a \fLG formula $A$ its positive polarization $\FtoT[\PF]{A}$ (resp. negative polariazation $\FtoT[\NF]{A}$) is a positive (resp. negative) \fDLG-formula defined by \eqref{eq:polar-form}. Moreover, we have that $\FtoT[\PF]{A}$ (resp. $\FtoT[\NF]{A}$) is pure iff $A$ is \fLG-positive (resp. \fLG-negative).

\begin{equation}\label{eq:polar-form}
\arraycolsep=0.7mm
\begin{array}{rclrclrclrcl}
     \FtoT[\PF]{A \mand B} &=& \FtoT[\PF]{A} \mand \FtoT[\PF]{B} &
     \FtoT[\PF]{A \mdlarr B} &=& \FtoT[\PF]{A} \mdlarr \FtoT[\NF]{B} &
     \FtoT[\PF]{A \mdrarr B} &=& \FtoT[\NF]{A} \mdrarr \FtoT[\PF]{B} &
     \FtoT[\PF]{p} &=& p \\
     
     \FtoT[\PF]{A \mor B} &=& \srdar (\FtoT[\NF]{A} \mor \FtoT[\NF]{B}) &
     \FtoT[\PF]{A \mrarr B} &=& \srdar (\FtoT[\PF]{A} \mrarr \FtoT[\NF]{B}) &
     \FtoT[\PF]{A \mlarr B} &=& \srdar (\FtoT[\NF]{A} \mlarr \FtoT[\PF]{B}) &
     \FtoT[\PF]{n} &=& \srdar n \\[2mm]
     
     \FtoT[\NF]{A \mand B} &=& \sbuar (\FtoT[\PF]{A} \mand \FtoT[\PF]{B}) &
     \FtoT[\NF]{A \mdlarr B} &=& \sbuar (\FtoT[\PF]{A} \mdlarr \FtoT[\NF]{B}) &
     \FtoT[\NF]{A \mdrarr B} &=& \sbuar (\FtoT[\NF]{A} \mdrarr \FtoT[\PF]{B}) &
     \FtoT[\NF]{p} &=& \sbuar p \\
     
     \FtoT[\NF]{A \mor B} &=& \FtoT[\NF]{A} \mor \FtoT[\NF]{B} &
     \FtoT[\NF]{A \mrarr B} &=& \FtoT[\PF]{A} \mrarr \FtoT[\NF]{B} &
     \FtoT[\NF]{A \mlarr B} &=& \FtoT[\NF]{A} \mlarr \FtoT[\PF]{B} &
     \FtoT[\NF]{n} &=& n \\
\end{array}
\end{equation}
\end{definition}

\begin{definition}\label{def:trans-str-seq}
Given an input structure $X$ (resp. an output structure $\Delta$), $\FtoT[\PS]{X}$ (resp. $\FtoT[\NS]{\Delta}$) is a positive structure (resp. negative structure) of \fDLG defined by \eqref{eq:trans-str} without structural shift or $\ell,r$-variant. The translation of a \fLG-sequent into a \fDLG-sequent is given by \eqref{eq:trans-seq}.

\begin{equation}\label{eq:trans-str}
\arraycolsep=0.7mm
\begin{array}{rclrclrclrcl}
     \FtoT[\PS]{X \MAND Y} &=& \FtoT[\PS]{X} \MAND \FtoT[\PS]{Y} &
     \FtoT[\PS]{X \MDLARR \Delta} &=& \FtoT[\PS]{X} \MDLARR \FtoT[\NS]{\Delta} &
     \FtoT[\PS]{\Delta \MDRARR Y} &=& \FtoT[\NS]{\Delta} \MDRARR \FtoT[\PS]{Y} &
     \FtoT[\PS]{A} &=& \FtoT[\PF]{A} \\
     
     \FtoT[\NS]{\Delta \MOR \Gamma} &=& \FtoT[\NS]{\Delta} \MOR \FtoT[\NS]{\Gamma} &
     \FtoT[\NS]{X \MRARR \Delta} &=& \FtoT[\PS]{X} \MRARR \FtoT[\NS]{\Delta} &
     \FtoT[\NS]{\Delta \MLARR Y} &=& \FtoT[\NS]{\Delta} \MLARR \FtoT[\PS]{Y} &
     \FtoT[\NS]{A} &=& \FtoT[\NF]{A} \\
\end{array}
\end{equation}

\begin{equation}\label{eq:trans-seq}
\arraycolsep=1.5mm
\begin{array}{rclcrclcrcl}
     \FtoT[]{X \vdash \focus{A}} &=&
     \FtoT[\PS]{X} \rvD \FtoT[\PF]{A} &~~&
     
     \FtoT[]{\focus{A} \vdash \Delta} &=&
     \FtoT[\NF]{A} \bvD \FtoT[\NS]{\Delta} &~~&

    \FtoT[]{X \vdash \Delta} &=&
    \FtoT[\PS]{X} \vD \FtoT[\NS]{\Delta}
\end{array}
\end{equation}
\end{definition}

\begin{definition}
A \fDLG-sequent $S$ is called \textbf{normal} if there exists a \fLG-sequent $S'$ such that $S = \FtoT[]{S'}$.
\end{definition}

By the remark in definition \ref{def:trans-str-seq}, normal sequents do not contain structural shifts or $\ell,r$-variants. Therefore, a normal sequent is either positive focused, negative focused or neutral (non-focused).

\begin{theorem}\label{thm:trans}
Every \fLG-derivation of a sequent $S$ can be translated into a \fDLG-derivation of $\FtoT[]{S}$.
\end{theorem}

\begin{proof}
Set a proof $\pi$ of a sequent $S$. We translate every rule $R$ of $\pi$ by a rule $\FtoT[]{R}$ of \fDLG by induction on $\pi$. We only treat a sufficient sample of rules. The others are similar.

\begin{equation}\notag
\begin{array}{ccc}
    \left\lceil % \Ax
    \AXC{}
    \RL{Ax}
    \UIC{$p \vdash \focus{p}$}
    \DP \right\rceil 
& = &
    \AXC{}
    \RL{$Id$}
    \UIC{$p \rvd p$}
    \dashedLine
    \UIC{$\FtoT[\PS]{p} \rvd \FtoT[\PF]{p}$}
    \DP\\ &  & \\

    \left \lceil % Ton
    \AXC{$X \vdash \focus{A}$}
    \AXC{$Y \vdash \focus{B}$}
    \RL{$\mand_R$}
    \BIC{$X \MAND Y \vdash \focus{A \mand B}$}
    \DP \right\rceil
& =&
    \AXC{$\FtoT[\PS]{X} \rvD \FtoT[\PF]{A}$}
    \AXC{$\FtoT[\PS]{Y} \rvD \FtoT[\PF]{B}$}
    \RL{$\mand_R$}
    \BIC{$\FtoT[\PS]{X} \MAND \FtoT[\PS]{Y} \rvd
        \FtoT[\PF]{A} \mand \FtoT[\PF]{B}$}
    \dashedLine
    \UIC{$\FtoT[\PS]{X \MAND Y} \rvd
    \FtoT[\PF]{A \mand B}$}
    \DP\\ &  & \\
    
    \left \lceil % Trans
    \AXC{$A \MAND B \vdash \Delta$}
    \RL{$\mand_L$}
    \UIC{$A \mand B \vdash \Delta$}
    \DP \right\rceil
& =&
    \AXC{$\FtoT[\PS]{A \MAND B} \vD \FtoT[\NS]{\Delta}$}
    \dashedLine
    \UIC{$\FtoT[\PF]{A} \MAND \FtoT[\PF]{B} \vD \FtoT[\NS]{\Delta}$}
    \RL{$\mand_R$}
    \UIC{$\FtoT[\PF]{A} \mand \FtoT[\PF]{B} \vD \FtoT[\NS]{\Delta}$}
    \dashedLine
    \UIC{$\FtoT[\PS]{A \mand B} \vD \FtoT[\NS]{\Delta}$}
    \DP
\end{array}
\end{equation}

In the following, we assume that $A$ is \fLG-positive, i.e. begins by $\mand, \mdlarr, \mdrarr$ or a positive atom. From definition \ref{def:polar-form}, we clearly have that $\FtoT[\NF]{A} = \sbuar \FtoT[\PF]{A}$ with $\FtoT[\PF]{A}$ being positive pure.

\begin{equation}\notag
\begin{array}{ccc}
    \left \lceil % Focusing
    \AXC{$A \vdash \Delta$}
    \RL{$\tilde{\mu}$}
    \UIC{$\focus{A} \vdash \Delta$}
    \DP \right\rceil 
& = &
    \AXC{$\FtoT[\PF]{A} \vD \FtoT[\NS]{\Delta}$}
    \RL{$\SBUARR$}
    \UIC{$\SBUARR \FtoT[\PF]{A} \bvD \FtoT[\NS]{\Delta}$}
    \RL{$\sbuar_L$}
    \UIC{$\sbuar \FtoT[\PF]{A} \bvD \FtoT[\NS]{\Delta}$}
    \dashedLine
    \UIC{$\FtoT[\NF]{A} \bvD \FtoT[\NS]{\Delta}$}
    \DP \\ & & \\
    
    \left \lceil % Defocusing
    \AXC{$X \vdash \focus{A}$}
    \RL{$\mu^*$}
    \UIC{$X \vdash A$}
    \DP \right\rceil
& = &
    \AXC{$\FtoT[\PS]{X} \rvd \FtoT[\PF]{A}$}
    \RL{$\sbuar_R$}
    \UIC{$\SBUARR \FtoT[\PS]{X} \bvvvd \sbuar \FtoT[\PF]{A}$}
    \RL{$\SBUARR$}
    \UIC{$\FtoT[\PS]{X} \vvdt \sbuar \FtoT[\PF]{A}$}
    \dashedLine
    \UIC{$\FtoT[\PS]{X} \vvdt \FtoT[\NS]{A}$}
    \DP
\end{array}
\end{equation}

Display postulates are translated by themselves. Note that translation $\FtoT[]{\cdot}$ is injective on derivations.
\end{proof}

%%%%%
\subsection{From \fDLG to \fLG}

We establish the form of \fDLG-derivable normal sequents in proposition \ref{prop:form-der-norm-seq} to translate them back to \fLG. We identify proof-sections containing non-normal sequents in proposition \ref{prop:proc-proof-sections} and give their translation back to \fLG in theorem \ref{thm:trans-b}.

\begin{definition}
The depolarization $\TtoF{A}$ of a \fDLG-formula $A$ is the \fLG-formula obtained by removing the shifts of $A$.
\end{definition}

\begin{definition}
Given a \fDLG-structure $\Psi$ without structural shift or $\ell,r$-variant, its translation $\TtoF{\Psi}$, obtained by depolarizing each formula occurring in $\Psi$, is a \fLG-structure. Moreover, if $\Psi$ is positive (resp. negative), then $\TtoF{\Psi}$ is an input (resp. output) structure.
\end{definition}

\begin{definition}
Given a normal sequent $S$ of \eqref{eq:form-der-norm-seq}, its translation $\TtoF{S}$ is given by \eqref{eq:trans-b-seq}.

\begin{equation}\label{eq:trans-b-seq}
    \TtoF{\GX \rvD \GP} = \TtoF{\GX} \vdash \focus{\TtoF{\GP}} ~~~~~~~~
    \TtoF{\GN \bvD \GD} = \focus{\TtoF{\GN}} \vdash \TtoF{\GD} ~~~~~~~~
    \TtoF{\GX \vD \GD} = \TtoF{\GX} \vdash \TtoF{\GD}
\end{equation}
\end{definition}

\begin{proposition}
Every \fDLG-derivable variant-free sequent admits a minimal proof, i.e. a cut-free and variant-free proof where there is no sequence of a rule an its inverse and no shift display postulate.
\end{proposition}

\begin{proof}
\fDLG-cuts can be eliminated and we can reduce every proof by eliminating step by step sequences of a rule and its inverse (i.e. display postulates and $\SRDARR$ and $\SBUARR$). In particular, all rules including variants (and thus all sequents including variants) get eliminated. We also replace sequences of rules $\SBUARR, \SBUARR \dashv \SRDARR$ (resp. $\SRDARR, \SBUARR \dashv \SRDARR$) by $\SRDARR$ (resp. $\SBUARR$), so that there is no shift display postulate anymore. Therefore, a minimal proof does not contain any rule which premise(s) and conclusion are all non-neutral non-focused sequents.
\end{proof}

\begin{definition}
A processing proof-section of a \fDLG-proof $\pi$ is a proof-section of $\pi$ where the leaves and the root are normal sequents and there is at least one internal non-normal sequent.
\end{definition}

\begin{proposition}\label{prop:proc-proof-sections}
In a minimal proof, every processing proof-section is of the form of one proof-section of \eqref{eq:proc-proof-sections}.

\begin{equation}\label{eq:proc-proof-sections}
\arraycolsep=1.5pt
\begin{array}{cccc}
     \fel \left\{
     \AXC{$N \bvd \Delta$}
     \RL{$\srdar_L$}
     \UIC{$\srdar N \rvvvd \SRDARR \Delta$}
     \RL{$\SRDARR$}
     \UIC{$\srdar N \tvvd \Delta$}
     \DP \right.&
     &
     
     \fer \left\{
     \AXC{$X \rvd P$}
     \RL{$\sbuar_R$}
     \UIC{$\SBUARR X \bvvvd \sbuar P$}
     \RL{$\SBUARR$}
     \UIC{$X \vvdt \sbuar P$}
     \DP \right.&
     
     \\~\\
     \fir \left\{
     \AXC{$X \vd N$}
     \RL{$\SRDARR$}
     \UIC{$X \rvvd \SRDARR N$}
     \RL{$\srdar_R$}
     \UIC{$X \rvvd \srdar N$}
     \DP \right.&
     
     \dot{\fir} \left\{
     \AXC{$\SX \tvvd N$}
     \RL{$\SRDARR$}
     \UIC{$\SX \rvvvd \SRDARR N$}
     \RL{$\srdar_R$}
     \UIC{$\SX \rvvvd \srdar N$}
     \DP \right.&
     
     \fil \left\{
     \AXC{$P \vd \Delta$}
     \RL{$\SBUARR$}
     \UIC{$\SBUARR P \bvvd \Delta$}
     \RL{$\sbuar_L$}
     \UIC{$\sbuar P \bvvd \Delta$}
     \DP \right.&
     
     \dot{\fil} \left\{
     \AXC{$P \vvdt \SD$}
     \RL{$\SBUARR$}
     \UIC{$\SBUARR P \bvvvd \SD$}
     \RL{$\sbuar_L$}
     \UIC{$\sbuar P \bvvvd \SD$}
     \DP \right.
     
     \\~\\
    \felfir \left\{
    \AXC{$n \bvd n$}
    \RL{$\srdar_L$}
    \UIC{$\srdar n \rvvvd \SRDARR n$}
    \RL{$\srdar_R$}
    \UIC{$\srdar n \rvvvd \srdar n$}
    \DP \right. &
    &
    
    \ferfil \left\{
    \AXC{$p \rvd p$}
    \RL{$\sbuar_R$}
    \UIC{$\SBUARR p \bvvvd \sbuar p$}
    \RL{$\sbuar_L$}
    \UIC{$\sbuar p \bvvvd \sbuar p$}
    \DP \right. &
\end{array}
\end{equation}
\end{proposition}

\begin{proof}
By looking at the rules of \fDLG, the rules where the premise is normal but not the conclusion are $\srdar_L$, $\srdar_R$, $\sbuar_L$, $\sbuar_R$, $\SRDARR$ and $\SBUARR$. Minimality of the proof forbids to stay in positive / negative non-focused phase for two sequents or more. It is easy to check that the cases $\felfir$ and $\ferfil$ are only possible when both precedent and succedent are a formula, which only happens when that formula is an atom.
\end{proof}

\begin{theorem}\label{thm:trans-b}
Every \fDLG-derivation of a normal sequent $S$ can be transformed into a derivation of $\TtoF{S}$.
\end{theorem}

\begin{proof}
Set a derivation of a normal sequent $S$ and call $\pi$ its minimal proof. We translate every rule $R$ of $\pi$ by a rule $\TtoF{R}$ of \fLG by induction on $\pi$. As previously, we only treat a sufficient sample of the cases.

$\bullet$ Normal rules (i.e. both premises and conclusions are normal):

\begin{equation}\notag
\begin{array}{ccc}
     \left\lfloor % Id
     \AXC{}
     \RL{$p$-Id}
     \UIC{$p \rvd p$}
     \DP \right\rfloor
& =& 
    \AXC{}
    \RL{$Ax$}
    \UIC{$p \vdash \focus{p}$}
    \dashedLine
    \UIC{$\TtoF{p} \vdash \focus{\TtoF{p}}$}
    \DP \\ ~\\
    
    \left\lfloor % Ton
    \AXC{$\GX \rvD \GP$}
        \AXC{$\GY \rvD \GQ$}
    \RL{$\mand_R$}
    \BIC{$\GX \MAND \GY \rvd \GP \mand \GQ$}
    \DP\right\rfloor
& =&
    \AXC{$\TtoF{\GX} \vdash \focus{\TtoF{\GP}}$}
        \AXC{$\TtoF{\GY} \vdash \focus{\TtoF{\GQ}}$}
    \RL{$\mand_R$}
    \BIC{$\TtoF{\GX} \MAND \TtoF{\GY} 
        \vdash \focus{\TtoF{\GP} \mand \TtoF{\GQ}}$}
    \dashedLine
    \UIC{$\TtoF{\GX \MAND \GY} 
        \vdash \focus{\TtoF{\GP \mand \GQ}}$}
    \DP \\ ~\\
    
     \left\lfloor % Trans
     \AXC{$\GP \MAND \GQ \vD \GD$}
     \RL{$\mand_L$}
     \UIC{$\GP \mand \GQ \vD \GD$}
     \DP \right\rfloor
& =&
    \AXC{$\TtoF{\GP} \MAND \TtoF{\GQ} \vdash \TtoF{\GD}$}
    \RL{$\mand_L$}
    \UIC{$\TtoF{\GP} \mand \TtoF{\GQ}  \vdash \TtoF{\GD}$}
    \dashedLine
    \UIC{$\TtoF{\GP \mand \GQ} \vdash \TtoF{\GD}$}
    \DP
\end{array}
\end{equation}
Display postulates of LG connectives are translated by themselves too.

$\bullet$ We translate processing proof-sections by one or two (de)focusing rules. Here, we take the case $P$ is pure positive, so $\TtoF{P}$ is \fLG-positive and the application of $\mu^*$ and $\tilde{\mu}$ is allowed.

\begin{equation}\notag
\begin{array}{ccc}
     \left\lfloor % defocusing
     \fer \left\{
     \AXC{$X \rvd P$}
     \RL{$\sbuar_R$}
     \UIC{$\SBUARR X \bvvvd \sbuar P$}
     \RL{$\SBUARR$}
     \UIC{$X \vvdt \sbuar P$}
     \DP \right. \right\rfloor
& =&
     \AXC{$\TtoF{X} \vdash \focus{\TtoF{P}}$}
     \RL{$\mu^*$}
     \UIC{$\TtoF{X} \vdash \TtoF{P}$}
     \dashedLine
     \UIC{$\TtoF{X} \vdash \TtoF{\sbuar P}$}
     \DP \\~ \\
     
     \left\lfloor % focusing
     \fil \left\{
     \AXC{$P \vd \Delta$}
     \RL{$\SBUARR$}
     \UIC{$\SBUARR P \bvvd \Delta$}
     \RL{$\sbuar_L$}
     \UIC{$\sbuar P \bvvd \Delta$}
     \DP \right. \right\rfloor
& =&
     \AXC{$\TtoF{P} \vdash \TtoF{\Delta}$}
     \RL{$\tilde{\mu}$}
     \UIC{$\focus{\TtoF{P}} \vdash \TtoF{\Delta}$}
     \dashedLine
     \UIC{$\focus{\TtoF{\sbuar P}} \vdash \TtoF{\Delta}$}
     \DP \\~\\
     
     \left\lfloor % de- + focusing
     \ferfil \left\{
    \AXC{$p \rvd p$}
    \RL{$\sbuar_R$}
    \UIC{$\SBUARR p \bvvvd \sbuar p$}
    \RL{$\sbuar_L$}
    \UIC{$\sbuar p \bvvvd \sbuar p$}
    \DP \right. \right\rfloor
& =&
     \AXC{$\TtoF{p} \vdash \focus{\TtoF{p}}$}
     \RL{$\mu^*$}
     \UIC{$\TtoF{p} \vdash \TtoF{p}$}
     \RL{$\tilde{\mu}$}
     \UIC{$\focus{\TtoF{p}} \vdash \TtoF{p}$}
     \DP \\~\\
\end{array}
\end{equation}

Although translation $\TtoF{\cdot}$ is not injective on structures, it is injective on derivations.
\end{proof}

%%%%%%%%%%%%
\section{Conclusions}

We observe that every connective in the language of \fDLG exhibits a core of minimal properties in any sub-algebra, namely it has finite arity and it is residuated in each coordinate. This leaves open the option that additional properties hold in the full algebra. Special sub-classes of $\FPLG$ algebras could then be captured by expanding the minimal calculus with opportune structural rules. 
If the rules are \emph{analytic-inductive} (see \cite{GMPTZ} for a definition) we conjecture that the cut-elimination will be preserved too. What we have in mind here is a natural generalisation of the cut-elimination meta-theorem of multi-type display calculi for a broader class of calculi, of which \fDLG is a prototypical example. We also plan to investigate up to which extent focalization will be preserved too.

% We observe that focalization also has a semantic footprint. The presence of weakening relations allows us to distinguish the personality of LG connectives in the sub-algebras from their personality \emph{from the point of view of weakening relations}. We also remark that this split is only relevant when shifts can be non-isomorphic. Such distinctions of behaviour are probably the key to keep semantic focalization while adding more structural properties. Special sub-classes of $\FPLG$ algebras could then be captured by expanding the minimal calculus with opportune structural rules.

}

%%%%%%%%%%
%\section{Conclusions}
%\label{sec:conclusion}

%\bibliographystyle{plainurl}
%\bibliography{bibliography}

\appendix
%%%%%%%%%%

{\renewcommand{\L}{\mathcal{L}}
%%%%%%%%%%

%%%%%%%%%%

%%%%%%%%%%
\section{Proof complements}\label{ap:proofs}

\begin{proof}[Proof of proposition \ref{prop:wr}]
Set $A_1, A_2 \in \A$, and $B_1', B_2' \in \B'$ such that $A_2 \leq_{\A} A_1$, $B_1' \leq_{\B'} B_2'$ and $A_1 \wrel B_1$. Equation $A_1 \wrel B_1$ is equivalent to $L(A_1) \wrel_{\B} B_1'$, and as $\wrel_{\B}$ is a weakening relation, we have $L(A_1) \wrel_{\B} B_2'$. This last equation is equivalent (by adjunction) to $A_1 \wrel_{\A} R(B_2')$ and as $\wrel_{\A}$ is a weakening relation, we have $A_2 \wrel_{\A} R(B_2')$, which is equivalent to $A_2 \wrel B_2'$. Therefore, $\wrel$ is a weakening relation.
\end{proof}

\begin{proof}[Proof that $\hvD$ is a weakening relation (definition \ref{def:collage-wr})]
We have $\hvD =\, \hvvd \sqcup \hvd \sqcup \vvdh$ a relation on $(\A \sqcup \A') \times (\B \sqcup \B')$. Set $A_1, A_2 \in \A \sqcup \A'$ and $B_1, B_2 \in \B \sqcup \B'$ such that $A_2 \leq_{\A \sqcup \A'} A_1$, $B_1 \leq_{\B \sqcup \B'} B_2$ and $A_1 \hvD B_1$. We only show how to get $A_2 \hvD B_1$, the other side is similar by symmetry of the problem.

\begin{itemize}
    \item If $A_1$ and $A_2$ are both in $\A$ (resp. $\A'$), then $A_2 \leq_{\A \sqcup \A'} A_1$ is equivalent to $A_2 \leq_{\A} A_1$ (resp. $A_2 \leq_{\A'} A_1$) and $A_1 \hvD B_1$ is equivalent to $A_1 \vvdh B_1$ or $A_1 \hvd B_1$ (resp. $A_1 \hvvd B_1$). In all cases, because $\hvvd, \hvd$ and $\vvdh$ are all weakening relations, we have $A_2 \hvvd B_1$, $A_2 \hvd B_1$ or $A_2 \vvdh B_1$, hence $A_2 \hvD B_1$.
    
    \item If $A_2 \in \A$ and $A_1 \in \A'$, then we have $A_2 \wrel_{\A} A_1$ and $A_1 \hvvd B_1$. Therefore $A_2 \hvd B_1$, because $\wrel_{\A} \hvvd \subseteq \hvd$ by hypothesis, hence $A_2 \hvD B_1$.
    
    \item The case $A_2 \in \A'$ and $A_1 \in \A$ is impossible because the union is disjoint and $\wrel_{\A}$ is only from $\A$ to $\A'$.
\end{itemize}
\end{proof}

The following four lemmas provide detailed complements to the proof of completeness by showing that the equivalence relations $\approx_s$ for $s \in \{\P, \sP, \N, \sN\}$ are congruences are thus that the operations and weakening relations / orders defined on them are well-defined.

\begin{lemma}\label{lem:prop-ftom}
Given a formula $A$ and a structure $\Psi$, we write $\Str(A)$ the structure obtained by turning the connectives of $A$ into their structural counterparts, and $\Form(\Psi)$ the formula obtained by turning the connectives of $\Psi$ into their operational counterpart, when it exists (in the other case, $\Form(\Psi)$ is not defined). When $\Ftom{~}$ and $\FtoM{~}$ are defined, they enjoy the following property:
\begin{equation}\notag\arraycolsep=1mm
\begin{array}{c}
\begin{array}{cl}
     \Ftom{A} = \Ftom{\Str(A)} & \text{if $A$ is a $\F$-formula} \\
     \Ftom{\Psi} = \Form(\Psi) & \text{if $\Psi$ is a $\G$-structure} \\
\end{array} \\
\begin{array}{ccc}
    \Ftom{\GX \MAND \GY} = \Ftom{\GX} \MAND \Ftom{\GY} &
    \Ftom{\GX \MDLARR \GD} = \Ftom{\GX} \MDLARR \FtoM{\GD} &
    \Ftom{\GD \MDRARR \GY} = \FtoM{\GD} \MDRARR \Ftom{\GY}
    \\
    \Ftom{\GG \MANDL \GY} = \Ftom{\GG} \MANDL \Ftom{\GY} &
    \Ftom{\GG \MDLARRL \GD} = \Ftom{\GG} \MDLARRL \FtoM{\GD} &
    \Ftom{\GX \MDRARRL \GY} = \FtoM{\GX} \MDRARRL \Ftom{\GY}
    \\
    \Ftom{\GX \MANDR \GG} = \Ftom{\GX} \MANDR \Ftom{\GG} &
    \Ftom{\GX \MDLARRR \GY} = \Ftom{\GX} \MDLARRR \FtoM{\GY} &
    \Ftom{\GD \MDRARRR \GG} = \FtoM{\GD} \MDRARRR \Ftom{\GG}
    \\
    \Ftom{\SBUARR X} = \SBUARR \Ftom{X} &&
    \Ftom{\BUARR \SX} = \BUARR \Ftom{\SX}
\end{array}\\[3mm]
\begin{array}{cl}
     \FtoM{A} = \FtoM{\Str(A)} & \text{if $A$ is a $\G$-formula} \\
     \FtoM{\Psi} = \Form(\Psi) & \text{if $\Psi$ is a $\F$-structure} \\
\end{array} \\
\begin{array}{ccc}
    \FtoM{\GD \MOR \GG} = \FtoM{\GD} \MOR \FtoM{\GG} &
    \FtoM{\GX \MRARR \GD} = \Ftom{\GX} \MRARR \FtoM{\GD} &
    \FtoM{\GD \MLARR \GY} = \FtoM{\GD} \MLARR \Ftom{\GY}
    \\
    \FtoM{\GY \MORL \GG} = \FtoM{\GY} \MORL \FtoM{\GG} &
    \FtoM{\GG \MRARRL \GD} = \Ftom{\GG} \MRARRL \FtoM{\GD} &
    \FtoM{\GX \MLARRL \GY} = \FtoM{\GX} \MLARRL \Ftom{\GY}
    \\
    \FtoM{\GD \MORR \GY} = \FtoM{\GD} \MORR \FtoM{\GY} &
    \FtoM{\GX \MRARRR \GY} = \Ftom{\GX} \MRARRR \FtoM{\GY} &
    \FtoM{\GD \MLARRR \GG} = \FtoM{\GD} \MLARRR \Ftom{\GG}
    \\
    
    \FtoM{\SRDARR \Delta} = \SRDARR \FtoM{\Delta} &&
    \FtoM{\RDARR \SD} = \RDARR \FtoM{\SD}
\end{array} \\[3mm]
    \Ftom{p} = \FtoM{p} = p ~~~~~~~~ \Ftom{n} = \FtoM{n} = n
\end{array}
\end{equation}
\end{lemma}

\begin{proof}
Unfolding definition \ref{def:ftom} directly gives these results.
\end{proof}

\begin{lemma}\label{lem:trans-rules}
If $\GX \vD \Delta$ (resp. $X \vD \GD$) is derivable, then
    \AXC{$\GX \vD \Delta$}
    \UIC{$\GX \vD \Form(\Delta)$}
    \DP
    \begin{math}
    \left(resp.~
    \AXC{$X \vD \GD$}
    \UIC{$\Form(X) \vD \GD$}
    \DP \right)
    \end{math}
is derivable.
\end{lemma}

\begin{proof}
By induction on $\Delta$ (resp. $X$), by successively applying translation rules on all structural connectives of $\Delta$ (resp. X).
\end{proof}

\begin{lemma}\label{lem:ax-on-str}
For every structure $\Psi$ of sort $s$, the sequent $\Ftom{\Psi} \ t_s\ \FtoM{\Psi}$ is derivable, if it is defined.
\end{lemma}

\begin{proof}
By induction on $\Psi$. If $\Psi$ is an atomic formula, $\Ftom{\Psi} = \FtoM{\Psi}$ so $p$-Id or $n$-Id is applicable.

If $\Psi = \GX \MAND \GY$, let us develop the case $\GX = X$ and $\GY = \SY$, the others being similar. We use the induction hypothesis on $X$ and $\SY$. By assumption, $\Ftom{\SY}$ exists, so $\SY$ does begin by $\MORL, \MORR, \MRARRL, \MRARRR, \MLARRL$ or $\MLARRR$. So it begins by a shift: $\SY = \SRDARR \Delta$ and we can then derive

\begin{center}
    \AXC{}
    \LL{(IH)$_X$}
    \UIC{$\Ftom{X} \rvd \FtoM{X}$}
    \dashedLine
    \LL{Lemma \ref{lem:prop-ftom}}
    \UIC{$\Ftom{X} \rvd \Form(X)$}
        \AXC{}
        \RL{(IH)$_{\SY}$}
        \UIC{$\Ftom{\SY} \rvvvd \FtoM{\SRDARR \Delta}$}
        \dashedLine
        \RL{Lemma \ref{lem:prop-ftom}}
        \UIC{$\Ftom{\SY} \rvvvd \SRDARR \FtoM{\Delta}$}
        \RL{$\SRDARR$}
        \UIC{$\Ftom{\SY} \tvvd \FtoM{\Delta}$}
        \RL{Lemma \ref{lem:trans-rules}}
        \UIC{$\Ftom{\SY} \tvvd \Form(\Delta)$}
        \RL{$\SRDARR$}
        \UIC{$\Ftom{\SY} \rvvvd \SRDARR \Form(\Delta)$}
        \RL{$\srdar_R$}
        \UIC{$\Ftom{\SY} \rvvvd \srdar \Form(\Delta)$}
        \dashedLine
        \UIC{$\Ftom{\SY} \rvvvd \Form(\SRDARR \Delta)$}
    \RL{$\mand_R$}
    \BIC{$\Ftom{X} \MAND \Ftom{\SY} \rvd \Form(X) \mand \Form(\SRDARR \Delta)$}
    \dashedLine
    \RL{Lemma \ref{lem:prop-ftom}}
    \UIC{$\Ftom{X \MAND \SY} \rvd \FtoM{X \MAND \SRDARR \Delta}$}
    \DP
\end{center}

The other LG-connectives work similarly.

If $\Psi = \SRDARR \Delta$, we use the induction hypothesis on $\Delta$ to get

\begin{center}
    \AXC{$\Ftom{\Delta} \bvd \FtoM{\Delta}$}
    \dashedLine
    \RL{Lemma \ref{lem:prop-ftom}}
    \UIC{$\Form(\Delta) \bvd \FtoM{\Delta}$}
    \LL{$\srdar_L$}
    \UIC{$\srdar \Form(\Delta) \rvvvd \SRDARR \FtoM{\Delta}$}
    \dashedLine
    \RL{Lemma \ref{lem:prop-ftom}}
    \UIC{$\Ftom{\SRDARR \Delta} \rvvvd \FtoM{\SRDARR \Delta}$}
    \DP
\end{center}
$\SBUARR$ works dually.

If $\Psi$ begins by a $\ell,r$-variants or shift adjoint, $\Psi$ is not in the domain of both $\Ftom{~}$ and $\FtoM{~}$.
\end{proof}

\begin{lemma}\label{lem:cut-on-str}
For every derivable sequents $\Ftom{\Psi} \t \FtoM{\Phi}$ and $\Ftom{\Phi} \ t'\ \FtoM{\Psi'}$, we can derive the cut
\begin{center}
    \AXC{$\Ftom{\Psi} \t \FtoM{\Phi}$}
    \AXC{$\Ftom{\Phi} \ t'\ \FtoM{\Psi'}$}
    \BIC{$\Ftom{\Psi} \ tt'\ \FtoM{\Psi'}$}
    \DP
\end{center}
where the composition $tt'$ is determined by the sort of $\Psi$ and $\Psi'$.
\end{lemma}

\begin{proof}
We proceed by induction on $\Phi$.

If $\Phi$ is an atomic formula, $\Ftom{\Phi} = \FtoM{\Phi}$ is a formula, so we can proceed to a $tt'$ cut of \eqref{eq:(co)axioms-and-cut-rules} (i.e. P-Cut, N-Cut, Pn-Cut or nN-Cut).

If $\Phi = \GX \MAND \GY$, we know that at the introduction of $\FtoM{\Psi} = \Form(\GX) \mand \Form(\GY)$, we have some sequent $\GX' \MAND \GY' \rvd \Form(\GX) \mand \Form(\GY)$ and the proof

\def\fCenter{}
\begin{center}
    \AX$\vdots\fCenter~\pi_1$
    \noLine
    \UI$\GX' \ \,\nrvD\fCenter\ \, \FtoM{\GX}$
        \AX$\vdots\fCenter~\pi_2$
        \noLine
        \UI$\GY' \ \,\nrvD\fCenter\ \, \FtoM{\GY}$
    \RL{$\mand_R$}
    \BI$\GX' \MAND \GY' \ \,\nrvd\fCenter\ \, \FtoM{\GX} \mand \FtoM{\GY}$
    \noLine
    \UI$\vdots\fCenter~\pi$
    \noLine
    \UI$\Ftom{\Psi} \ t\fCenter\ \FtoM{\Phi}$
    \DP
\end{center}

We can then apply the induction hypothesis on $\GX$ and $\GY$. Here, we develop case where $t' = \nrvD$ (so $\Psi$ is positive):

\begin{center}
    \AX$\vdots\fCenter~\pi_1$
    \noLine
    \UI$\GX' \ \,\nrvD\fCenter\ \, \FtoM{\GX}$
        \AX$\vdots\fCenter~\pi_2$
        \noLine
        \UI$\GY' \ \,\nrvD\fCenter\ \, \FtoM{\GY}$
            \AX$\vdots\fCenter$
            \noLine
            \UI$\Ftom{\GX \MAND \GY} \ \nrvD\fCenter\ \FtoM{\Psi'}$
            \dashedLine
            \RL{Lemma \ref{lem:prop-ftom}}
            \UI$\Ftom{\GX} \MAND \Ftom{\GY} \, \ \nrvD\fCenter\ \, \FtoM{\Psi'}$
            \RL{$\MAND \dashv \MRARRR$}
            \UI$\Ftom{\GY} \, \ \nrvD\fCenter\ \, \Ftom{\GX} \MRARRR \FtoM{\Psi'}$
        \RL{(IH)$_{\GY}$}
        \BI$\GY' \,\ \nrvD\fCenter\ \,  \Ftom{\GX} \MRARRR \FtoM{\Psi'}$
        \RL{$\MAND \dashv \MRARRR$}
        \UI$\Ftom{\GX} \MAND \Ftom{\GY} \,\ \nrvD\fCenter\ \, \FtoM{\Psi'}$
        \RL{$\MAND \dashv \MLARRL$}
        \UI$\Ftom{\GX} \,\ \nrvD\fCenter\ \, \FtoM{\Psi'} \MLARRL \GY'$
    \RL{(IH)$_{\GX}$}
    \BI$\GX' \,\ \nrvD\fCenter\ \, \FtoM{\Psi'} \MLARRL \GY'$
    \RL{$\MAND \dashv \MLARRL$}
    \UI$\GX' \MAND \GY' \ \,\nrvD\fCenter\ \, \FtoM{\Psi'}$
    \noLine
    \UI$\vdots\fCenter~\pi[\FtoM{\Psi'} / \FtoM{\Phi}]$
    \noLine
    \UI$\Ftom{\Psi} \ tt'\fCenter\ \FtoM{\Psi'}$
    \DP
\end{center}

The case where $t' = \nvD$ works similarly with $\MRARR$ and $\MLARR$.

The fact that the uniform substitution $\pi[\FtoM{\Psi'} / \FtoM{\Phi}]$ (were $\FtoM{\Psi'}$ may have a different sort from $\FtoM{\Phi}$) can be defined and is derivable in \fDLG is not proven here. We would have to provide a way of transforming some structural connectives (in specific positions) into others, what is partly already implicit in the use of overloaded connectives (e.g. $\MAND$'s arguments can be either pure or shifted). This operation pertains to the problem of canonical cut-elimination with heterogeneous sequents, what is left to a subsequent paper.

If $\Phi = \SBUARR X$, we use the same procedure, by induction hypothesis on $X$. The turnstile $t'$ can only be $\nbvD$. Here we develop the case where $\Psi'$ is shifted.

\begin{center}
    \AX$\vdots\fCenter~\pi'$
    \noLine
    \UI$X' \ \,\nrvd\fCenter\ \, \FtoM{X}$
    \RL{$\mand_R$}
    \UI$\SBUARR X' \ \,\nbvvvd\fCenter\ \, \sbuar \FtoM{X}$
    \noLine
    \UI$\vdots\fCenter~\pi$
    \noLine
    \UI$\Ftom{\Psi} \ t\fCenter\ \FtoM{\Phi}$
    \DP $~~~~\rightsquigarrow~~~~$
    \AX$\vdots\fCenter~\pi'$
    \noLine
    \UI$X' \ \,\nrvD\fCenter\ \, \FtoM{X}$
        \AX$\vdots\fCenter$
        \noLine
        \UI$\Ftom{\SBUARR X} \ \nbvvvd\fCenter\ \FtoM{\Psi'}$
        \dashedLine
        \RL{Lemma \ref{lem:prop-ftom}}
        \UI$\SBUARR \Ftom{X} \, \ \nbvvvd\fCenter\ \, \FtoM{\Psi'}$
        \RL{$\SBUARR \dashv \RDARR$}
        \UI$\Ftom{X} \, \ \nrvd\fCenter\ \, \RDARR \FtoM{\Psi'}$
    \RL{(IH)$_{X}$}
    \BI$X' \,\ \nrvd\fCenter\ \, \RDARR \FtoM{\Psi'}$
    \RL{$\SBUARR \dashv \RDARR$}
    \UI$\SBUARR X' \ \,\nbvvvd\fCenter\ \, \FtoM{\Psi'}$
    \noLine
    \UI$\vdots\fCenter~\pi[\FtoM{\Psi'} / \FtoM{\Phi}]$
    \noLine
    \UI$\Ftom{\Psi} \ tt'\fCenter\ \FtoM{\Psi'}$
    \DP
\end{center}
The other cases are treated similarly.
\end{proof}

\begin{proof}[Complements for the proof of completeness (theorem \ref{thm:completeness})]
We prove that the equivalence relations $\approx_{s}$ are congruences, i.e. that they respect the operations and the orders. It will follow that the operations and weakening relations defined on these equivalence classes are well-defined. We only detail the case of $\mand$ with one pure and one shifted premise, and $\hvd$, the rest being similar.

\begin{center}
    \AXC{$X \approx_{\P} Y$}
    \dashedLine
    \UIC{$\Ftom{X} \rvd \FtoM{Y}$}
        \AXC{$\SX \approx_{\sP} \SY$}
        \dashedLine
        \UIC{$\Ftom{\SX} \rvvvd \FtoM{\SY}$}
    \LL{$\mand_R$}
    \BIC{$\Ftom{X} \MAND \Ftom{\SX} \rvd \FtoM{Y} \mand \FtoM{\SY}$}
    \dashedLine
    \LL{Lemma \ref{lem:prop-ftom}}
    \UIC{$\Ftom{X \MAND \SX} \rvd \FtoM{Y \MAND \SY}$}
    
    \AXC{$X \approx_{\P} Y$}
    \dashedLine
    \UIC{$\Ftom{Y} \rvd \FtoM{X}$}
        \AXC{$\SX \approx_{\sP} \SY$}
        \dashedLine
        \UIC{$\Ftom{\SY} \rvvvd \FtoM{\SX}$}
    \RL{$\mand_R$}
    \BIC{$\Ftom{Y} \MAND \Ftom{\SY} \rvd \FtoM{X} \mand \FtoM{\SX}$}
    \dashedLine
    \RL{Lemma \ref{lem:prop-ftom}}
    \UIC{$\Ftom{Y \MAND \SY} \rvd \FtoM{X \MAND \SX}$}
    
    \dashedLine
    \BIC{$X \MAND \SX \approx_{\P} Y \MAND \SY$}
    \DP

    \AXC{$X \approx_{\P} Y$}
    \dashedLine
    \UIC{$\Ftom{X} \rvd \FtoM{Y}$}
        \AXC{$\Ftom{Y} \vd \FtoM{\Gamma}$}
            \AXC{$\Delta \approx_{\N} \Gamma$}
            \dashedLine
            \UIC{$\Ftom{\Gamma} \bvd \FtoM{\Delta}$}
        \RL{Lemma \ref{lem:cut-on-str}}
        \BIC{$\Ftom{Y} \vd \FtoM{\Delta}$}
    \LL{Lemma \ref{lem:cut-on-str}}
    \BIC{$\Ftom{X} \vd \FtoM{\Delta}$}
    \DP
\end{center}

For every homogeneous turnstile $t$, reflexivity,\footnote{Reflexivity on structures $\Psi$ such that $\Ftom{\Psi} \t \FtoM{\Psi}$ is not defined is explicitly added.} transitivity and antisymmetry of $t^{\bbA}$ are a consequent of lemma \ref{lem:ax-on-str}, lemma \ref{lem:cut-on-str} and definition of $\approx_s$ respectively. For heterogeneous turnstiles $t$, the weakening property of $t^{\bbA}$ is a consequence of lemma \ref{lem:cut-on-str}. The property \eqref{eq:shifts-fplg} of definition \ref{def:fplg} is due to to rules $\SBUARR \dashv \RDARR$ and $\BUARR \dashv \SRDARR$ and \eqref{eq:shift-intro-elim} to rules $\SRDARR$ and $\SBUARR$. The adjunction \eqref{eq:adj-fplg} and \eqref{eq:adj-fplg-variants} straightforwardly hold thanks to the corresponding rules in \eqref{eq:display-postulates}. We only develop the example of property $\sbuar^{\bbA} \dashv \rdar^{\bbA}$:
\begin{center}
    \AXC{$\sbuar^{\bbA} [X]_{\approx_{\P}} \ \,\nbvvvd^{\bbA}\ \, [\SD]_{\approx_{\sN}}$}
    \dashedLine
    \UIC{$\Ftom{\SBUARR X} \bvvvd \FtoM{\SD}$}
    \dashedLine
    \RL{Lemma \ref{lem:prop-ftom}}
    \UIC{$\SBUARR \Ftom{X} \bvvvd \FtoM{\SD}$}
    \RL{$\SBUARR \dashv \RDARR$}
    \UIC{$\Ftom{X} \rvd \RDARR \FtoM{\SD}$}
    \dashedLine
    \RL{Lemma \ref{lem:prop-ftom}}
    \UIC{$\Ftom{X} \rvd \FtoM{\RDARR \SD}$}
    \dashedLine
    \UIC{$[X]_{\approx_{\P}} \ \,\nrvd^{\bbA}\ \, \rdar^{\bbA}  [\SD]_{\approx_{\sN}}$}
    \DP
\end{center}
\end{proof}

\section{Symmetries}\label{ap:symmetries}

Lambek-Grishin calculus exhibits two main symmetries \cite{Moortgat:2009}: an order-preserving left-right symmetry $\cdot^{\bowtie}$ and an order-reversing dual symmetry $\cdot^{\infty}$ represented in \eqref{eq:symmetries-1}\footnote{These definitions should be understood as $(A \mand B)^{\bowtie} = B^{\bowtie} \mand A^{\bowtie}$, $(A \mand B)^{\infty} = B^{\infty} \mor A^{\infty}$, etc.}. We extend them to $\FPLG$ and $\fDLG$ by \eqref{eq:symmetries-2}. The dual of a turnstile $t$ is given by \eqref{eq:symmetries-2} through the turnstile interpretation of \eqref{eq:turnstile-interpretation}: $t^{\infty} = (t^{\bbA})^{\infty}$, e.g. $\nrvd^{\infty} = \nbvd$.

\begin{equation}\label{eq:symmetries-1}
\begin{array}{cccc}
\bowtie &
\begin{array}{cccc}
    A \mrarr C & A \mand B & A \mor B & C \mdlarr B \\ \hline\hline
    C \mlarr A & B \mand A & B \mor A & B \mdrarr C
\end{array}& ~~~~~~~~

\infty &
\begin{array}{ccc}
    A \mrarr C & A \mand B & C \mlarr B \\ \hline\hline
    C \mdlarr A & B \mor A & B \mdrarr C
\end{array}
\end{array}
\end{equation}

% For the sake of conciseness, we sometimes refer to these symmetries to avoid spelling out all cases.

\begin{equation}\label{eq:symmetries-2}
\begin{array}{c}
\bowtie ~~~~
\begin{array}{cccccc}
    \GM \mrarrl \GN & \GQ \mrarrr \GP & \GN \mandl \GP &
    \GN \morr \GP & \GN \mdlarrl \GM & \GP \mdlarrr \GQ \\ \hline\hline
    
    \GN \mlarrr \GM & \GP \mlarrl \GQ & \GP \mandr \GN &
    \GP \morl \GN & \GM \mdrarrr \GN & \GQ \mdrarrl \GP
\end{array}\\[4mm]

\infty ~~~~
\begin{array}{ccccccccccc}
    A \mrarrl C & A \mrarrr C &
    A \mandl B & A \mandr B &
    C \mlarrl B & C \mlarrr B &
    \arvd & \hrvvd & \arvvvd & \vvdh & \hvd \\ \hline\hline
    
    C \mdlarrr A &  C \mdlarrl A &
    B \morr A & B \morl A &
    B \mdrarrr C & B \mdrarrl C &
    \abvd & \hbvvd & \abvvvd & \hvvd & \hvd
\end{array}
\end{array}
\end{equation}

The presentation of \fDLG rules in section \ref{subsec:fDLG} also reflects the dual symmetry. In equation (\ref{eq:operational-rules}) the dual of rule $R$ is the one displayed on the opposite side of the page w.r.t.~the vertical axis. Therefore we have the following property.

\begin{proposition}
If $\Phi \t \Psi$ is a derivable sequent, then $\Phi^{\bowtie} \t \Psi^{\bowtie}$ and $\Psi^{\infty} \ t^{\infty}\  \Phi^{\infty}$ are also derivable.
\end{proposition}

\section{Examples}\label{ap:examples}

In the tradition of parsing-as-deduction \cite{lambek1958mathematics,lam61}, various extensions of the Lambek calculus have been proposed to recognize whether sentences are syntactically well-formed and tell apart different readings \cite{Moortgat--Moot:2011}. A well-formed sentence like \ref{ex:example}, for example, is semantically ambiguous as shown by the paraphrases \ref{ex:for-ex} or \ref{ex:ex-for}.
% MM: structural ambiguity is something else

\ex. \a.\label{ex:example} Everyone likes some teacher.
    \b.\label{ex:for-ex} For everyone, there is a teacher such that she likes it. \hfill ($\forall \exists$-reading)
    \c. \label{ex:ex-for} There is a teacher such that everyone likes it. \hfill ($\exists \forall$-reading)

In particular, the previous readings can be already captured by genuinely different focused proofs of the minimal Lambek logic. In figure \ref{fig:example-derivation-ex-fa} and figure \ref{fig:example-derivation-fa-ex} we provide \fDLG-derivations of the two readings, where
every occurrence of atoms $n$ (common nouns) and $np$ (noun phrases) is assigned a positive polarity, and every occurrence of atoms $s$ (sentence) is assigned a negative polarity (cfr.~derivation (35) in \cite{Moortgat--Moot:2011} where the polarity assignment is called `bias'). Notice how in figure \ref{fig:example-derivation-ex-fa} first \texttt{some teacher} is attacked, and only then \texttt{everyone} is attacked. In figure \ref{fig:example-derivation-fa-ex} is the other way around. Figure \ref{fig:ex-generation-tree} shows the signed generation of the end-sequent.

\begin{figure}[t]
    \centering
\begin{tikzpicture}
\draw (0,0) node{
\Tree[.$\boxed{+\MAND}$
        [.$\boxed{+~\mand}$
            [.$+\srdar$ [.$+\mlarr$ [.$\boxed{+\sbuar}$ $\boxed{+ np}$ ]
                $-n$ ] ] 
            $\boxed{+n}$ ]
        [.$\boxed{+\MAND}$
            [.$+\srdar$ [.$+\mlarr$
                            [.$+\mrarr$ $-np$ $+s$ ]
                            $-np$ ] ]
            [.$\boxed{+\MAND}$
                [.$+\srdar$ [.$+\mlarr$ [.$\boxed{+\sbuar}$ $\boxed{+np}$ ]
                    $-n$ ] ]
                $\boxed{+n}$ ] ] ]
};
\draw (5,2.5) node{$\rvd$};
\draw (7,2.5) node{\Tree[.$\boxed{-\srdar}$ $\boxed{-s}$ ]};
\end{tikzpicture}
    \caption{\ \ Signed generation tree of the end-sequent in figure \ref{fig:example-derivation-fa-ex} and \ref{fig:example-derivation-ex-fa}. Skeleton nodes are encapsulated in a box where PIA nodes are not.}
    \label{fig:ex-generation-tree}
\end{figure}

\begin{figure}[t]
    \centering
{\fns
\AXC{$np \rvd  np$}

\AXC{$s \bvd s$}

\LL{$\mrarr_L$}
\BIC{$np \mrarr s \bvd np \MRARR s$}

\AXC{$np \rvd  np$}

\LL{$\mrarr_L$}
\BIC{$(np \mrarr s) \mlarr np \bvd (np \MRARR s) \MLARR np$}

\dashedLine
\UIC{$\likes \bvd (np \MRARR s) \MLARR np$}

\LL{$\srdar_L$}
\UIC{$\srdar \likes \rvvvd  \SRDARR\, ((np \MRARR s) \MLARR np)$}

\RL{$\SRDARR$}
\UIC{$\srdar\, \likes \tvvd (np \MRARR s) \MLARR np$}

\LL{Display}
\dashedLine
\UIC{$np \vd \srdar\, \likes \MRARR (np \MRARR s)$}

\LL{$\SBUARR$}
\UIC{$\SBUARR\, np \bvvd \srdar\, \likes \MRARR (np \MRARR s)$}

\LL{$\sbuar_L$}
\UIC{$\sbuar\, np \bvvd \srdar\, \likes \MRARR (np \MRARR s)$}

\AXC{$n \rvd n$}

\dashedLine
\UIC{$\teacher \rvd n$}

\LL{$\mlarr_L$}
\BIC{$\sbuar\, np \mlarr n \bvd (\srdar\, \likes \MRARR (np \MRARR s)) \MLARR \teacher$}

\dashedLine
\UIC{$\some \bvd (\srdar\, \likes \MRARR (np \MRARR s)) \MLARR \teacher$}

\LL{$\srdar_L$}
\UIC{$\srdar \some \rvvvd \SRDARR\, ((\srdar\, \likes \MRARR (np \MRARR s)) \MLARR \teacher)$}

\RL{$\SRDARR$}
\UIC{$\srdar\, \some \tvvd  (\srdar\, \likes \MRARR (np \MRARR s)) \MLARR \teacher$}

\LL{Display}
\dashedLine
\UIC{$np \vd s \MLARR (\srdar\, \likes \MAND (\srdar \, \some \MAND \teacher))$}

\LL{$\SBUARR$}
\UIC{$\SBUARR np \bvvd s \MLARR (\srdar\, \likes \MAND (\srdar \, \some \MAND \teacher))$}

\LL{$\sbuar_L$}
\UIC{$\sbuar\, np \bvvd s \MLARR (\srdar\, \likes \MAND (\srdar \, \some \MAND \teacher))$}

\AXC{$n \rvd n$}

\dashedLine
\UIC{$\one \rvd n$}

\LL{$\mlarr_L$}
\BIC{$\sbuar\, np \mlarr n \bvd (s \MLARR (\srdar\, \likes \MAND (\srdar \, \some \MAND \teacher))) \MLARR \one$}

\dashedLine
\UIC{$\every \bvd (s \MLARR (\srdar\, \likes \MAND (\srdar \, \some \MAND \teacher))) \MLARR \one$}

\LL{$\srdar_L$}
\UIC{$\srdar\, \every \rvvvd \SRDARR\, ((s \MLARR (\srdar\, \likes \MAND (\srdar \, \some \MAND \teacher))) \MLARR \one)$}

\RL{$\SRDARR$}
\UIC{$\srdar\, \every \tvvd (s \MLARR (\srdar\, \likes \MAND (\srdar \, \some \MAND \teacher))) \MLARR \one$}

\LL{Display}
\dashedLine
\UIC{$\srdar\, \every \MAND \one \vd s \MLARR (\srdar\, \likes \MAND (\srdar \, \some \MAND \teacher))$}

\LL{$\mand_L$}
\UIC{$\srdar\, \every \mand \one \vd s \MLARR (\srdar\, \likes \MAND (\srdar \, \some \MAND \teacher))$}

\dashedLine
\UIC{$\everyone \vd s \MLARR (\srdar\, \likes \MAND (\srdar \, \some \MAND \teacher))$}

\LL{$\MAND \dashv \MLARR$}
\UIC{$\everyone \MAND (\srdar\, \likes \MAND (\srdar \, \some \MAND \teacher)) \vd s$}

\RL{$\SRDARR$}
\UIC{$\everyone \MAND (\srdar\, \likes \MAND (\srdar \, \some \MAND \teacher)) \rvvd \SRDARR\, s$}

\RL{$\srdar_R$}
\UIC{$\everyone \MAND  (\srdar\, \likes \MAND (\srdar \, \some \MAND \teacher)) \rvvd \srdar\, s$}

\DP 
}
    \caption{\ \ Derivation attached to \ref{ex:for-ex}: $\forall\exists$-reading.}
    \label{fig:example-derivation-fa-ex}
\end{figure}

%%%
\begin{figure}[t!]
    \centering
{\fns
\AXC{$np \rvd  np$}

\AXC{$s \bvd s$}

\LL{$\mrarr_L$}
\BIC{$np \mrarr s \bvd np \MRARR s$}

\AXC{$np \rvd  np$}

\LL{$\mrarr_L$}
\BIC{$(np \mrarr s) \mlarr np \bvd (np \MRARR s) \MLARR np$}

\dashedLine
\UIC{$\likes \bvd (np \MRARR s) \MLARR np$}

\LL{$\srdar_L$}
\UIC{$\srdar \likes \rvvvd  \SRDARR\, ((np \MRARR s) \MLARR np)$}

\RL{$\SRDARR$}
\UIC{$\srdar\, \likes \tvvd (np \MRARR s) \MLARR np$}

\LL{Display}
\dashedLine
\UIC{$np \vd s \MLARR (\srdar\, \likes \MAND np)$}

\LL{$\SBUARR$}
\UIC{$\SBUARR\, np \bvvd s \MLARR (\srdar\, \likes \MAND np)$}

\LL{$\sbuar_L$}
\UIC{$\sbuar\, np \bvvd s \MLARR (\srdar\, \likes \MAND np)$}

\AXC{$n \rvd n$}

\dashedLine
\UIC{$\one \rvd n$}

\LL{$\mlarr_L$}
\BIC{$\sbuar\, np \mlarr n \bvd (s \MLARR (\srdar\, \likes \MAND np)) \MLARR \one$}

\dashedLine
\UIC{$\every \bvd (s \MLARR (\srdar\, \likes \MAND np)) \MLARR \one$}

\LL{$\srdar_L$}
\UIC{$\srdar \every \rvvvd \SRDARR\, ((s \MLARR (\srdar\, \likes \MAND np)) \MLARR \one)$}

\RL{$\SRDARR$}
\UIC{$\srdar\, \every \tvvd  (s \MLARR (\srdar\, \likes \MAND np)) \MLARR \one$}

\LL{Display}
\dashedLine
\UIC{$np \vd \likes \MRARR ((\srdar\, \every \MAND \one) \MRARR s)$}

\LL{$\SBUARR$}
\UIC{$\SBUARR np \bvvd \likes \MRARR ((\srdar\, \every \MAND \one) \MRARR s)$}

\LL{$\sbuar_L$}
\UIC{$\sbuar\, np \bvvd \likes \MRARR ((\srdar\, \every \MAND \one) \MRARR s)$}

\AXC{$n \rvd n$}

\dashedLine
\UIC{$\teacher \rvd n$}

\LL{$\mlarr_L$}
\BIC{$\sbuar\, np \mlarr n \bvd (\likes \MRARR ((\srdar\, \every \MAND \one) \MRARR s)) \MLARR \teacher$}

\dashedLine
\UIC{$\some \bvd (\likes \MRARR ((\srdar\, \every \MAND \one) \MRARR s)) \MLARR \teacher$}

\LL{$\srdar_L$}
\UIC{$\srdar\, \some \rvvvd \SRDARR\, ((\likes \MRARR ((\srdar\, \every \MAND \one) \MRARR s)) \MLARR \teacher)$}

\RL{$\SRDARR$}
\UIC{$\srdar\, \some \tvvd (\likes \MRARR ((\srdar\, \every \MAND \one) \MRARR s)) \MLARR \teacher$}

\LL{Display}
\dashedLine
\UIC{$\srdar\, \every \MAND \one \vd s \MLARR (\srdar\, \likes \MAND (\srdar \, \some \MAND \teacher))$}

\LL{$\mand_L$}
\UIC{$\srdar\, \every \mand \one \vd s \MLARR (\srdar\, \likes \MAND (\srdar \, \some \MAND \teacher))$}

\dashedLine
\UIC{$\everyone \vd s \MLARR (\srdar\, \likes \MAND (\srdar \, \some \MAND \teacher))$}

\LL{$\MAND \dashv \MLARR$}
\UIC{$\everyone \MAND (\srdar\, \likes \MAND (\srdar \, \some \MAND \teacher)) \vd s$}

\RL{$\SRDARR$}
\UIC{$\everyone \MAND (\srdar\, \likes \MAND (\srdar \, \some \MAND \teacher)) \rvvd \SRDARR\, s$}

\RL{$\srdar_R$}
\UIC{$\everyone \MAND  (\srdar\, \likes \MAND (\srdar \, \some \MAND \teacher)) \rvvd \srdar\, s$}

\DP 
}
    \caption{\ \ Derivation attached to \ref{ex:ex-for}: $\exists\forall$-reading.}
    \label{fig:example-derivation-ex-fa}
\end{figure}

}

\end{document}